\documentclass[a4paper,11pt]{amsart}

\usepackage{amsmath,amssymb,amsthm,latexsym,amscd,mathrsfs}
\usepackage{indentfirst}
\usepackage{stmaryrd}
\usepackage{graphicx}
\usepackage{extarrows}
\usepackage{multirow}
\usepackage{latexsym}
\usepackage{amsfonts}
\usepackage{xcolor}
\usepackage{pictexwd,dcpic}
\usepackage{graphicx}
\usepackage{psfrag}
\usepackage{hyperref}
\usepackage{comment}
\usepackage{enumerate}
\usepackage{dutchcal}

\numberwithin{equation}{section} \theoremstyle{plain}
\newtheorem{thm}{Theorem}[section]

\newtheorem{lem}[thm]{Lemma}
\newtheorem{cor}[thm]{Corollary}
\newtheorem{defn}[thm]{Definition}

\makeatletter

\newcommand{\trace}{\mathrm{Tr}}
\newcommand{\Ric}{\mathrm{Ric}}
\newcommand{\Hn}{\mathcal{H}^{\mathrm{N}}}
\newcommand{\Ht}{\mathcal{H}^{\mathrm{T}}}

\makeatother

\title{Normal-Yang-Mills and Tangent-Yang-Mills submanifolds}

\author[J.Q. Ge]{Jianquan Ge}
\address{School of Mathematical Sciences, Laboratory of Mathematics and Complex Systems, Beijing Normal University, Beijing 100875, P.R. CHINA.}
\email{jqge@bnu.edu.cn}

\author[L.X. Xiao]{Lixin Xiao\textsuperscript{*}}
\address{Chern Institute of Mathematics and LPMC, Nankai University, Tianjin 300071, P. R. China}
\email{lixinxiao@nankai.edu.cn}

\author[W.J. Zhang]{Wenjin Zhang}
\address{Chern Institute of Mathematics and LPMC, Nankai University, Tianjin 300071, P. R. China}
\email{wjzhang@nankai.edu.cn}

\thanks {$^{*}$ the corresponding author.}
\thanks{J. Q. Ge is partially supported by NSFC (No. 12571049) and the Fundamental Research Funds for the Central Universities.}

\begin{document}
\maketitle


\begin{abstract}
This paper investigates the variational problems associated with the squared $L^2$-norms of the normal and Riemannian curvature tensors for submanifolds immersed in a unit sphere. We define the critical points of these functionals under normal variations as Normal-Yang-Mills and Tangent-Yang-Mills submanifolds, for which we explicitly establish the Euler-Lagrange equations in terms of the second fundamental form. Furthermore, by investigating the focal submanifolds of OT-FKM isoparametric hypersurfaces, we construct infinitely many non-trivial examples of both Normal-Yang-Mills and Tangent-Yang-Mills submanifolds. Notably, the curvature tensors of these examples generally do not satisfy the classical Yang-Mills equations.
\end{abstract}

\section{Introduction}

Let $E$ be a vector bundle with a metric over an $n$-dimensional Riemannian manifold $M^n$. In classical Yang-Mills theory \cite{B-L, F-U, R-S-S}, the Yang-Mills functional is defined as the squared $L^2$-norm of the curvature $\Omega$ associated with a connection $\nabla$ on $E$:
\begin{equation*}
    \mathscr{Y\!\!M}(\nabla) = \int_{M} \| \Omega \|^2 dvol.
\end{equation*}
Here, the variation is taken with respect to the connection $\nabla$. The critical points of this functional, termed Yang-Mills connections, are governed by the well-known first variational formulas:
\begin{itemize}
    \item \textbf{Integral form:} $\int_{M}\langle d_{\nabla}\tau, \Omega \rangle\,dvol=0$ for any $\mathfrak{so}(E)$-valued $1$-form $\tau$, where $d_{\nabla}$ denotes the exterior covariant derivative with respect to $\nabla$ on $E$.
    \item \textbf{Differential form:} $\delta_{\nabla}\Omega=0$, where $\delta_{\nabla}$ is the co-differential operator.
\end{itemize}

When $M$ is immersed into the unit sphere $S^{n+p}$, its geometry is characterized by two fundamental curvature tensors: the Riemannian curvature $\Omega$ of the tangent bundle and the normal curvature $\Omega^{\bot}$ of the normal bundle. The relationship between classical Yang-Mills connections and submanifold geometry was established by Tian \cite{T} in $2000$, who demonstrated that Yang-Mills connections are closely related to minimal submanifolds. This result motivated further research into curvature functionals on submanifolds. For instance, Chen and Zhou \cite{C-Z} initiated the study of $p$-Yang-Mills fields in $2007$, defining the functional as the $L^p$-norm of the connection curvature and establishing gap properties. Subsequently, in $2012$, Jia and Zhou \cite{J-Z} introduced a broader class of functionals to investigate $F$-Yang-Mills fields, focusing on the conditions under which certain vanishing theorems hold.

While the aforementioned studies focus on variations of the connection, an alternative variational problem for submanifolds treats the immersion itself as the variable. Sakamoto \cite{S} applied this approach to the normal curvature functional, deriving the first variation formula and expressing the resulting Euler-Lagrange equations in terms of divergence operators acting on tensor fields. In this paper, to bridge this variational problem with classical Yang-Mills theory, we investigate the curvature functionals defined by the squared $L^2$-norms of the Riemannian curvature and normal curvature:
\begin{equation*}
L = \int_{M} \|\Omega\|^2 dvol, \quad L^{\bot} = \int_{M} \|\Omega^{\bot}\|^2 dvol.
\end{equation*}

For both functionals, we establish explicit Euler-Lagrange equations in terms of the second fundamental form.

\begin{defn}
Let $M^n$ be a submanifold of $S^{n+p}$ with the immersion $F_{0}:M\to S^{n+p}$. A smooth map $F: M \times I \to S^{n+p}$, where $I=(-\epsilon, \epsilon)$ for some $\epsilon > 0$, is called a variation of $M$ if $F_{0}=F(\cdot,0)$, and $F_{t}:=F(\cdot,t)$ is an immersion from $M$ to $S^{n+p}$ for each $t\in I$. Moreover, $F$ is called a normal variation if its variation vector field $\frac{\partial F}{\partial t}\big|_{t=0}$ is normal to $M$. We assume that the initial immersion $F_0$ satisfies $L(F_0)<\infty$ and $L^\perp(F_0)<\infty$. All variations in this paper are assumed to be uniformly compactly supported, i.e., there exists a compact set \(K\) contained in the interior of \(M\) such that \(F_t=F_0\) on \(M\setminus K\) for all \(t\in I\), so that we can integrate by parts regardless of whether \(M\) is closed.
\end{defn}

\begin{defn}
Let $M^n$ be a submanifold of $S^{n+p}$.

(1) $M$ is called a Normal-Yang-Mills submanifold if $\frac{d}{dt}\big|_{t=0}L^{\bot}(F_{t})=0$ for all normal variations.

(2) $M$ is called a Tangent-Yang-Mills submanifold if $\frac{d}{dt}\big|_{t=0}L(F_{t})=0$ for all normal variations.
\end{defn}

Let $\{\theta_{i}\}_{i=1}^{n}$ be a local orthonormal coframe, $\{e_i\}_{i=1}^{n}$ be the dual local tangent frame and $\{e_{\alpha}\}_{\alpha=n+1}^{n+p}$ be a local normal orthonormal frame. The second fundamental form is $h=\sum\limits_{\alpha,i,j}h_{ij}^{\alpha}\theta_i \otimes \theta_j e_{\alpha}$,  where $h^{\alpha}=\sum\limits_{i,j}h_{ij}^{\alpha}\theta_i \otimes e_j $ is the shape operator in direction $e_{\alpha}$, and the mean curvature vector is $H=\sum\limits_{\alpha,i}h_{ii}^{\alpha}e_{\alpha}$. Let $\nabla,\nabla^\perp$ denote the Levi-Civita connection and the normal connection respectively, and let $D$ denote the corresponding connection on the vector bundles induced by $TM$ and $NM$. We define two normal vector fields $\Hn$ and $\Ht$ along the immersion:
\begin{equation*}
\begin{aligned}
    \Hn :=\sum_{\alpha}\left(\sum_{\beta,i,k} D_{e_i}\left(h^{\beta}_{ik}  (\delta_{\nabla^{\bot}}\Omega^{\bot})_{\alpha\beta}(e_{k})\right)+\sum_{\beta,\gamma}\trace(h^{\alpha}h^{\beta}h^{\gamma}h^{\beta}h^{\gamma}-h^{\alpha}h^{\beta}h^{\gamma}h^{\gamma}h^{\beta})\right)e_{\alpha},
\end{aligned}
\end{equation*}
and
\begin{multline*}
    \Ht :=\sum_{\alpha}\Biggl(\sum_{i,j,k}h^{\alpha}_{ij}\nabla_k (\nabla_{k}\Ric_{ij}-\nabla_{j}\Ric_{ik})-2\sum_{\beta}\trace(h^{\alpha}h^{\beta}h^{\beta})\\
    +\sum_{\beta,\gamma}\trace(h^{\alpha}h^{\beta}h^{\gamma})\cdot\trace(h^{\beta}h^{\gamma})-\sum_{\beta,\gamma}\trace(h^{\alpha}h^{\beta}h^{\gamma}h^{\beta}h^{\gamma})\Biggr)e_{\alpha}.
\end{multline*}

Utilizing these notions, the first variations of the curvature functionals are determined as follows.

\begin{thm} \label{thm-1}
Let $M^n$ be a submanifold of $S^{n+p}$. For a normal variation $F$, the first variation corresponding to the normal curvature functional $L^{\bot}$ is
\begin{align*}
    \frac{d}{dt}\bigg|_{t=0}L^{\bot}(F_{t})=\int_{M}\left\langle \frac{\partial F}{\partial t}\bigg|_{t=0},-4\Hn-\|\Omega^{\bot}\|^{2}H \right\rangle dvol.
\end{align*}
Consequently, $M$ is a Normal-Yang-Mills submanifold if and only if it satisfies the Euler-Lagrange equation: 
\begin{equation*}
4\Hn+\|\Omega^{\bot}\|^{2}H=0.
\end{equation*} 

In particular, if $M$ is a minimal submanifold, this condition reduces to $\Hn=0$.
\end{thm}

If the normal bundle is flat, i.e., $\Omega^{\bot}=0$, then the shape operators commute with each other and thus $\Hn=0$. The next corollary follows from
the theorem above.
\begin{cor}
A submanifold $M^n$ of $S^{n+p}$ with flat normal bundle is a Normal-Yang-Mills submanifold. In particular, any hypersurface $M^n$ of $S^{n+1}$ is Normal-Yang-Mills.
\end{cor}

\begin{thm} \label{thm-2}
Let $M^n$ be a submanifold of $S^{n+p}$. For a normal variation $F$, the first variation of the curvature functional $L$ is
\begin{align*}
    \frac{d}{dt}\bigg|_{t=0}L(F_{t})=\int_{M}\left\langle \frac{\partial F}{\partial t}\bigg|_{t=0}, 4\Ht + \Big(4(n-1) - \|\Omega\|^{2}\Big)H+8\sum_{\alpha,\beta} H^\beta \trace(h^\alpha h^\beta) e_\alpha \right\rangle dvol.
\end{align*}
Consequently, $M$ is a Tangent-Yang-Mills submanifold if and only if it satisfies the Euler-Lagrange equation: 
\begin{equation*}
4\Ht + \Big(4(n-1) - \|\Omega\|^{2}\Big)H+8\sum_{\alpha,\beta} H^\beta \trace(h^\alpha h^\beta) e_\alpha =0.
\end{equation*} 

In particular, if $M$ is a minimal submanifold, this condition reduces to $\Ht=0$.
\end{thm}

Our definitions of Normal-Yang-Mills and Tangent-Yang-Mills submanifolds naturally follow a well-established tradition of characterizing special submanifolds via the critical points of geometric functionals. The most fundamental example is the volume functional, whose critical points yield minimal submanifolds. Another prominent example is the study of Willmore submanifolds in the unit sphere by Wang \cite{W}, which are defined as the critical points of the Willmore functional $W(M) = \int_{M} (\|h\|^2 - \frac{1}{n}\|H\|^2)^{\frac{n}{2}} dvol$. For minimal submanifolds in a sphere with constant scalar curvature, this variational problem beautifully reduces to a cubic algebraic system $\sum\limits_{\beta} \trace (h^{\alpha} h^{\beta} h^{\beta}) = 0$. In this context, isoparametric focal submanifolds in spheres provide a rich source of examples. Specifically, Tang and Yan \cite{T-Y 1}, as well as Xie \cite{X}, proved that all such focal submanifolds, particularly those of OT-FKM type, are Willmore submanifolds. By examining these OT-FKM type focal submanifolds under our newly established criteria, we obtain richer geometric characterizations. More importantly, as detailed in the following theorems, these non-trivial examples demonstrate that our functionals introduce a distinct class of critical points whose curvatures generally fail to satisfy the classical Yang-Mills equations.

\begin{thm}\label{exmp-3}
Let $M_{+}$ be an OT-FKM type focal submanifold in a unit sphere $S^{2l-1}$ with multiplicities $(m_1,m_2)$.

(1) The normal connection on the normal bundle is Yang-Mills if and only if $(m_{1},m_{2})\in\{(2,1),(6,1)\}$, or $M_+$ is diffeomorphic to $Sp(2)$ with $(m_{1},m_{2})=(4,3)$.

(2) For $m_1=1, 2, 3$, the definite case of $m_1=4$, and the indefinite case of $(m_{1},m_{2})=(4,3)$,  $M_{+}$ is a Normal-Yang-Mills submanifold. In contrast, in the indefinite case of $m_1=4$ and $m_2>3$, $M_{+}$ is not a Normal-Yang-Mills submanifold.
\end{thm}

\begin{thm} \label{exmp-4}
Let $M_{+}$ be an OT-FKM type focal submanifold in a unit sphere $S^{2l-1}$ with multiplicities $(m_1,m_2)$.

(1) The Levi-Civita connection is Yang-Mills if and only if  $(m_{1},m_{2})\in\{(2,1),(6,1)\}$, or $M_+$ is diffeomorphic to $Sp(2)$ with $(m_{1},m_{2})=(4,3)$.

(2) For $m_1=1, 2, 3$, the definite case of $m_1=4$, and the indefinite case of $(m_{1},m_{2})=(4,3)$,  $M_{+}$ is a Tangent-Yang-Mills submanifold. In contrast, in the indefinite case of $m_1=4$ and $m_2>3$, $M_{+}$ is not a Tangent-Yang-Mills submanifold.
\end{thm}

We will show that it is equivalent for OT-FKM type focal submanifolds $M_+$ to be Normal-Yang-Mills and Tangent-Yang-Mills determined by the same criterion (\ref{sum distinct normal indices}). Notice that Theorems \ref{exmp-3} and \ref{exmp-4} provide infinitely many examples and counterexamples for Normal-Yang-Mills and Tangent-Yang-Mills minimal submanifolds in unit spheres.

The paper is organized as follows. In Section 2, we introduce the necessary notations and derive the first variations of the connections and curvatures. Section 3 is devoted to the proofs of Theorem \ref{thm-1} and Theorem \ref{thm-2}. In Section 4, we prove Theorems \ref{exmp-3} and \ref{exmp-4}, providing non-trivial examples of Normal-Yang-Mills and Tangent-Yang-Mills submanifolds whose respective curvature tensors do not satisfy the classical Yang-Mills equations, as well as counterexamples that fail to satisfy our new criteria.

\section{Notation and Local Formulas}
In this section, we introduce the necessary notation and recall fundamental equations to facilitate the study of the geometric properties of Normal-Yang-Mills and Tangent-Yang-Mills submanifolds.

Throughout this paper, we adopt the following index conventions:
\begin{align*}
    1\leqslant A, B, C \leqslant n+p, \quad 1\leqslant i, j, k, l \leqslant n, \quad n+1\leqslant \alpha, \beta, \gamma, \delta \leqslant n+p.
\end{align*}

Let $M$ be an $n$-dimensional smooth manifold, and let $F_0: M \to S^{n+p}$ be a smooth immersion. We consider a smooth variation of $F_0$, denoted by $F: M \times I \to S^{n+p}$, such that $F(\cdot, 0) = F_0$. For convenience, we denote $F_t = F(\cdot, t)$. Locally on an open neighborhood $U\times I$, we first choose a local orthonormal frame $\{e_{1},\dots,e_{n+p}\}$ on the pullback bundle $F^* TS^{n+p}$, such that at each point $(x,t)\in U\times I$, the basis $e_{1},\dots,e_{n}$ are tangent to $F_{t}(M)$ at point $F_{t}(x)$ and $e_{n+1},\dots,e_{n+p}$ are normal. Using the frame, we define local 1-forms on $U\times I$ by the following relations:
\begin{align*}
dF&=\sum_{A}\omega_{A}e_{A},\\
de_{A}&=\sum_{B}\omega_{AB}e_{B}-\omega_{A}F,
\end{align*}
or equivalently, by
\[\omega_{A}=\langle dF,e_{A}\rangle,\quad
\omega_{AB}=\langle de_{A},e_{B}\rangle.\]
Then differentiating the formulae yields the first and second structure equations:
\begin{align}
\label{first structure equation}
d\omega_{A} &= \sum_{B} \omega_{B} \wedge \omega_{BA},\\
\label{second structure equation}
d\omega_{AB} &= \sum_{C} \omega_{AC} \wedge \omega_{CB} - \omega_{A} \wedge \omega_{B}.
\end{align}

On $M \times I$, we decompose these 1-forms into two components:
\begin{equation}\label{decomp. of omega}
    \omega_A = \theta_A + a_A dt, \quad \omega_{AB} = \theta_{AB} + b_{AB} dt,
\end{equation}
where $\theta_A, \theta_{AB}$ are differential forms in $M$ with coefficients which depend on $t$. It is clear that all of $\omega_{AB},\ \theta_{AB}$, and $b_{AB}$ are anti-symmetric with respect to their indices. To see the geometric meaning of $a_{A}$, let $\omega_A$ act on the vector field generated by $t$ and we obtain
\[a_A=\omega_A\left(\frac{\partial}{\partial t}\right)=\left\langle dF\left(\frac{\partial}{\partial t}\right),e_A\right\rangle=\left\langle\frac{\partial F}{\partial t},e_A\right\rangle.\]
Thus for a normal variation $F$, the variation field $\frac{\partial F}{\partial t}\big|_{t=0}=\sum\limits_\alpha a_{\alpha}e_{\alpha}$.
Similarly, $b_{ij}=\left\langle\frac{\partial}{\partial t}e_{i},e_j\right\rangle$ and $b_{\alpha\beta}=\left\langle\frac{\partial}{\partial t}e_{\alpha},e_\beta\right\rangle$ represent the rotations of frames of the tangent bundle and the normal bundle respectively. However, as we will see in Lemma \ref{from structural equations}, the components $b_{i\alpha}$ are tensorial with respect to the indices and are determined by the variation $F$.

For each $t$, we consider the embedding $i_{t}:M\to M\times I,\ x\mapsto (x,t)$. Then the pullback $i_{t}^{*}\omega_{A}=i_{t}^{*}\theta_{A}$ and $i_{t}^{*}\omega_{AB}=i_{t}^{*}\theta_{AB}$ define a family of 1-forms with the parameter $t$. By a slight abuse of notation, we continue to denote the pullbacks of \(\theta_A\) and \(\theta_{AB}\) to \(M\) by the same symbols. Likewise, we use \(e_1,\ldots,e_n\) to denote the corresponding local orthonormal frame on \(M\) under the identification induced by \(dF_t\). Moreover, since $F\circ i_{t}=F_{t}$, the definitions of $\omega_{A}$ and $\omega_{AB}$ yield
\begin{align*}
\theta_{A}=i_{t}^{*}\omega_{A}=\langle dF_{t},e_{A}\rangle,\qquad \theta_{AB}=i_{t}^{*}\omega_{AB}=\langle \nabla^{S}e_{A},e_{B}\rangle,
\end{align*}
which are precisely the pullback forms from $S^{n+p}$, where $\nabla^{S}$ is the Levi-Civita connection on $S^{n+p}$. Consequently, $\theta_{i}$ forms the dual basis of $e_{i}$, and $\theta_{\alpha}$ vanishes identically. $\theta_{ij}$ and $\theta_{\alpha\beta}$ are connection forms on the tangent bundle $TM$ and the normal bundle $NM$ respectively; $\theta_{i\alpha}=\sum\limits_{j}h^{\alpha}_{ij}\theta_{j}$ is the second fundamental form. Denote by $\Omega_{ij}$ and $\Omega^{\bot}_{\alpha\beta}$ the Riemann curvature and the normal curvature forms, then the following structure equations and curvature equations are satisfied:
\begin{align}
&d_{M}\theta_{i}=\sum_{j}\theta_{j}\wedge\theta_{ji}, \nonumber \\[-0.5ex]
&d_{M}\theta_{ij}=\sum_{k}\theta_{ik}\wedge\theta_{kj}-\Omega_{ij}, \nonumber \\[-0.5ex]
&d_M\theta_{\alpha\beta}=\sum_{\gamma}\theta_{\alpha\gamma}\wedge\theta_{\gamma\beta}-\Omega^{\bot}_{\alpha\beta}, \nonumber \\[-0.5ex]
&\hspace{-1.2em}\left\{
\begin{alignedat}{2}\label{Gauss-Codazzi-Ricci}
&\mathrm{(Gauss)}\ &&\Omega_{ij}=\theta_{i}\wedge\theta_{j}+\sum_{\alpha}\theta_{i\alpha}\wedge\theta_{j\alpha},\\
&\mathrm{(Codazzi)}\ &&d_{M}\theta_{i\alpha}=\sum_{j}\theta_{ij}\wedge\theta_{j\alpha}+\sum_{\beta}\theta_{i\beta}\wedge\theta_{\beta\alpha},\\[-0.5ex]
&\mathrm{(Ricci)}\ &&\Omega^{\bot}_{\alpha\beta}=\sum_{i}\theta_{i\alpha}\wedge\theta_{i\beta}.
\end{alignedat}
\right.
\end{align}

\begin{lem}\label{from structural equations}
For a normal variation $F:M\times I\to S^{n+p}$, we have the following

(1) $\frac{\partial}{\partial t}\big|_{t=0}\theta_i=\sum\limits_j b_{ij}\theta_{j}-\sum\limits_\alpha a_{\alpha}\theta_{i\alpha}$;

(2) the covariant derivative $\nabla^{\bot} a_{\alpha}=\sum\limits_i b_{i\alpha}\theta_{i}$ on the initial submanifold $M$;

(3) the following frame-independent 1-forms $\tau=(\tau_{ij})\in\Lambda^1(\mathfrak{so}(TM))$ and $\tau^{\bot}=(\tau^{\bot}_{\alpha\beta})\in\Lambda^1(\mathfrak{so}(NM))$ measure the variation of connections caused by the variation $F$:
\begin{equation*}
\begin{aligned}
\tau_{ij} &=\left(\frac{\partial}{\partial t}\theta_{ij}-d_{\nabla}b_{ij}\right)\bigg|_{t=0}= \sum_{\alpha}(b_{i\alpha}\theta_{\alpha j}-\theta_{i\alpha}b_{\alpha j}),\\
\tau^{\bot}_{\alpha\beta} &=\left(\frac{\partial}{\partial t}\theta_{\alpha\beta}-d_{\nabla^{\bot}}b_{\alpha\beta}\right)\bigg|_{t=0}= \sum_{i}(b_{\alpha i}\theta_{i\beta}-\theta_{\alpha i}b_{i\beta}).
\end{aligned}
\end{equation*}

\end{lem}

\begin{proof}
(1) Substitute \eqref{decomp. of omega} into the first structure equation \eqref{first structure equation}, and use the decomposition of the differentiation $d=d_{M}+dt\wedge \frac{\partial}{\partial t}$, we have
\[\left(d_{M}+dt\wedge \frac{\partial}{\partial t}\right)(\theta_{A}+a_{A}dt)=\sum_B (\theta_B+a_{B}dt)\wedge(\theta_{BA}+b_{BA}dt).\]
Expand both sides and we derive
\begin{align*}
LHS&=d_{M}\theta_{A}+dt\wedge\left(\frac{\partial}{\partial t}\theta_{A}-d_{M}a_{A}\right),\\
RHS&=\sum_{B}\theta_{B}\wedge\theta_{BA}+dt\wedge\sum_{B}(a_{B}\theta_{BA}-\theta_{B}b_{BA}).
\end{align*}
It is easily seen that equating the component independent of $dt$ recovers the first structural equation on $M_{t}$ and the symmetry of its second fundamental form, i.e., $\sum\limits_{i}\theta_{i}\wedge\theta_{i\alpha}=0$. Meanwhile, comparing the components involving $dt$ on both sides yields
\begin{equation}\label{D_t theta_A}
\frac{\partial}{\partial t}\theta_{A}=d_{M}a_{A}+\sum_{B}(a_{B}\theta_{BA}-\theta_{B}b_{BA}).
\end{equation}
For tangent indices, since $\theta_{\alpha}=0$, we have
\begin{align*}
\frac{\partial}{\partial t}\theta_{i}
=d_{M}a_{i}+\sum_{j}(a_{j}\theta_{ji}-\theta_{j}b_{ji})+\sum_{\alpha}a_{\alpha}\theta_{\alpha i}.
\end{align*}
Our assumption of the normal variation implies $a_{i}=0$ when $t=0$, thus
\[\frac{\partial}{\partial t}\bigg|_{t=0}\theta_{i}
=\sum_{j}b_{ij}\theta_{j}-\sum_{\alpha}a_{\alpha}\theta_{i\alpha}.\]

(2) For normal indices,   \eqref{D_t theta_A} yields
\begin{align*}
0
&=d_{M}a_{\alpha}+\sum_{i}(a_{i}\theta_{i\alpha}-\theta_{i}b_{i\alpha})+\sum_{\beta}a_{\beta}\theta_{\beta\alpha}\\
&=\nabla^{\bot} a_{\alpha}+\sum_{i}(a_{i}\theta_{i\alpha}-\theta_{i}b_{i\alpha}),
\end{align*}
where $\nabla^{\bot} a_{\alpha}=d_{M}a_{\alpha}+\sum\limits_{\beta}a_{\beta}\theta_{\beta\alpha}$ is the covariant differential of the normal part of the variational field $\sum\limits_{\alpha} a_{\alpha}e_{\alpha}$. Thus,
\[b_{i\alpha}=\nabla_{e_{i}}^{\bot}a_{\alpha}+\sum_{j}a_{j}h_{ji}^{\alpha}\]
is a tensor determined by the variation of $M$. Specifically, $b_{i\alpha}=\nabla^\perp_{e_{i}}a_{\alpha}$ when $t=0$.

(3) Substituting the   \eqref{decomp. of omega} into the second structure equation \eqref{second structure equation} yields
\begin{align*}
d\omega_{AB}&=\sum_C (\theta_{AC}+b_{AC}dt)\wedge(\theta_{CB}+b_{CB}dt)-(\theta_{A}+a_{A}dt)\wedge(\theta_{B}+a_{B}dt)\\
&=\sum_C \theta_{AC}\wedge\theta_{CB}-\theta_{A}\wedge\theta_{B}+dt\wedge\left(\sum_{C}(b_{AC}\theta_{CB}-\theta_{AC}b_{CB})-(a_{A}\theta_{B}-\theta_{A}a_{B})\right),
\end{align*}
while the LHS
\begin{align*}
d\omega_{AB}=\left(d_{M}+dt\wedge \frac{\partial}{\partial t}\right)(\theta_{AB}+b_{AB}dt)=d_{M}\theta_{AB}+dt\wedge\left(\frac{\partial}{\partial t}\theta_{AB}-d_{M}b_{AB}\right).
\end{align*}

Therefore, the $dt$-independent part of the equation coincides with Gauss--Codazzi--Ricci equations \eqref{Gauss-Codazzi-Ricci} on $M$, while the terms involving $dt$ lead to
\begin{equation*}
\frac{\partial}{\partial t}\theta_{AB}=d_{M}b_{AB}+\sum_{C}(b_{AC}\theta_{CB}-\theta_{AC}b_{CB})-(a_{A}\theta_{B}-\theta_{A}a_{B}).
\end{equation*}

Consider the tangent indices, that is,
\[\frac{\partial}{\partial t}\theta_{ij}=d_{M}b_{ij}+\sum_{k}(b_{ik}\theta_{kj}-\theta_{ik}b_{kj})+\sum_{\alpha}(b_{i\alpha}\theta_{\alpha j}-\theta_{i\alpha}b_{\alpha j})-(a_{i}\theta_{j}-\theta_{i}a_{j}).\]

A variation of submanifold $F:M\times I\to S^{n+p}$ naturally induces a variation of the Levi-Civita connection on $M$, which is determined by a 1-form valued in the bundle $\mathfrak{so}(TM)$. Here, for a vector bundle $E$ with metric, $\mathfrak{so}(E)$ is a subbundle of End($E$) with all anti-symmetric elements. The time derivatives of the connection forms $\frac{\partial}{\partial t}\theta_{ij}$ are the most obvious choice; however, they rely on the choice of orthonormal frames. As we have stated before, $b_{ij}$ exactly represents the rotation of the frame $\{e_{i}\}$ on the tangent bundle $TM$. To eliminate this effect, we define the 1-form $\tau=(\tau_{ij})\in\Lambda^1(\mathfrak{so}(TM))$  by
\begin{equation*}\label{def of tau_ij}
\tau_{ij}=\left(\frac{\partial}{\partial t}\theta_{ij}-d_{\nabla}b_{ij}\right)\bigg|_{t=0}=\sum_{\alpha}(b_{i\alpha}\theta_{\alpha j}-\theta_{i\alpha}b_{\alpha j}),
\end{equation*}
which is independent of the frame since each term on the right-hand side is tensorial. Here, the operator $d_{\nabla}$ denotes the exterior covariant differential which satisfies that for any $\xi\in\Lambda^{q}(\mathfrak{so}(TM))$, $d_{\nabla}\xi=d_{M}\xi-[\nabla,\xi]$, where $\nabla=\theta=(\theta_{ij})$ is the connection form on $TM$. More specifically, for $b=(b_{ij})\in\Lambda^{0}(\mathfrak{so}(TM|_{U}))$,
\[d_{\nabla}b_{ij}=d_{M}b_{ij}-[\theta,b]_{ij}=d_{M}b_{ij}-\sum_{k}(\theta_{ik}b_{kj}-b_{ik}\theta_{kj}).\]

Analogously, we can define the frame-independent 1-form $\tau^{\bot}=(\tau^{\bot}_{\alpha\beta})\in\Lambda^1(\mathfrak{so}(NM))$ by
\begin{equation*}\label{def of tau_alpha beta}
\tau^{\bot}_{\alpha\beta}=\left(\frac{\partial}{\partial t}\theta_{\alpha\beta}-d_{\nabla^{\bot}}b_{\alpha\beta}\right)\bigg|_{t=0}=\sum_{i}(b_{\alpha i}\theta_{i\beta}-\theta_{\alpha i}b_{i\beta}),
\end{equation*}
where the operator $d_{\nabla^{\bot}}b^{\bot}=d_{M}b^{\bot}-[\theta^{\bot},b^{\bot}]$ denotes the exterior covariant differential with respect to the connection form $\nabla^{\bot}=\theta^{\bot}=(\theta_{\alpha\beta})$ on the normal bundle $NM$ and $b^{\bot}=(b_{\alpha\beta})\in\Lambda^{0}(\mathfrak{so}(NM|_{U}))$. 
\end{proof}

\begin{lem}\label{var of curvature forms}
Assume that $F: M \times I \to S^{n+p}$ is a normal variation. Then the variation of the curvature forms satisfies:
\begin{equation*}
\frac{\partial}{\partial t}\bigg|_{t=0}\Omega_{ij}=-d_{\nabla}\tau_{ij}+[b,\Omega]_{ij},\qquad \frac{\partial}{\partial t}\bigg|_{t=0}\Omega^{\bot}_{\alpha\beta}=-d_{\nabla^{\bot}}\tau^{\bot}_{\alpha\beta}+[b^{\bot},\Omega^{\bot}]_{\alpha\beta}.
\end{equation*}
\end{lem}

\begin{proof}
For simplicity, we omit the subscripts. Notice that the square of the exterior covariant differential leads to curvature, that is, for any $q$-form $\xi$ valued in $\mathfrak{so}(TM)$, $d_{\nabla}^{2}\xi=[\Omega,\xi]$. By Leibniz's rule, we obtain
\begin{align*}
\frac{\partial}{\partial t}\Omega&=\frac{\partial}{\partial t}(-d\theta+\theta\wedge\theta)\\
&=-d\left(\frac{\partial\theta}{\partial t}\right)+\frac{\partial\theta}{\partial t}\wedge\theta+\theta\wedge\frac{\partial\theta}{\partial t}\\
&=-d\left(\frac{\partial\theta}{\partial t}\right)+\left[\theta,\frac{\partial\theta}{\partial t}\right]\\
&=-d_{\nabla}\left(\frac{\partial\theta}{\partial t}\right).
\end{align*}
Consequently, from the definition of $\tau$,
\[\frac{\partial}{\partial t}\bigg|_{t=0}\Omega=-d_\nabla\left(\frac{\partial\theta}{\partial t}\bigg|_{t=0}\right)=-d_{\nabla}\tau-[\Omega,b].\]

Similarly, we can derive that $\frac{\partial}{\partial t}\big|_{t=0}\Omega^{\bot}_{\alpha\beta}=-d_{\nabla^{\bot}}\tau^{\bot}_{\alpha\beta}+[b^{\bot},\Omega^{\bot}]_{\alpha\beta}$.
\end{proof}

\section{Proof of Main Theorems}

In this section, we explore the correspondence between the Normal-Yang-Mills and Tangent-Yang-Mills submanifolds and classical Yang-Mills theory. And we derive the Euler-Lagrange equations for the curvature functionals and provide the proofs for our main results.  

\subsection{Classical Yang-Mills equations}

At any point $q \in M$, we can choose a local orthonormal frame $\{e_A\}$ such that the connection forms satisfy $\theta_{ij}(q)=0$ and $\theta_{\alpha\beta}(q)=0$.

Let $E$ be a vector bundle over $M$ equipped with a connection $\nabla$, and let $D$ be the natural tensor product connection on the tensor bundle $\Lambda^{q}T^{*}M\otimes E=\Lambda^{q}(E)$. For any $E$-valued $q$-form $\varphi\in\Lambda^{q}(E)$, the co-differential $\delta_{\nabla}$ is defined by
\begin{equation}\label{co-differential operator}
    (\delta_{\nabla}\varphi)_{X_{1},\dots,X_{q-1}}=-\sum\limits_{j=1}^{n}(D_{e_{j}}\varphi)_{e_{j},X_{1},\dots,X_{q-1}}.
\end{equation}
If $E$ is equipped with a metric, we choose a local orthonormal frame $\{s_{a}\}_{a=1}^{\mathrm{rank}E}$, then the inner product on $\mathfrak{so}(E)$ for any local sections $\varphi,\psi$ is
\[\langle\varphi,\psi\rangle=\sum_{a,b}\varphi_{ab}\psi_{ab}.\]
For $\varphi,\psi\in\Lambda^q(\mathfrak{so}(E))$, the inner product is defined by
\begin{align*}\label{product}
    \langle \varphi,\psi \rangle&=\sum\limits_{i_{1}<\cdots<i_{q}}\langle \varphi(e_{i_{1}},\dots, e_{i_{q}}), \psi(e_{i_{1}},\dots, e_{i_{q}}) \rangle_{\mathfrak{so}(E)}\\
    &=\sum_{a,b}\,\sum\limits_{i_{1}<\cdots<i_{q}} \varphi_{ab}(e_{i_{1}},\dots, e_{i_{q}})\psi_{ab}(e_{i_{1}},\dots, e_{i_{q}}) .
\end{align*} 

Based on this notation, the classical Yang-Mills equations can be characterized as follows.

\begin{lem}\label{codiff-riem-curv}
Let $M^n$ be a submanifold of $S^{n+p}$. The codifferential of the Riemannian curvature is given by the following identity.
    \begin{equation}\label{Yang-Mills equation for tangent}
        (\delta_{\nabla}\Omega)_{ij}(e_{k})  = -\nabla_{i}\Ric_{jk} + \nabla_{j}\Ric_{ik}.
    \end{equation}
\end{lem}

\begin{proof}
Since both sides of the identities are tensorial, it suffices to verify them at an arbitrary point $q \in M$ under a local orthonormal frame chosen such that $\theta_{ij}(q) = 0$. Consequently, applying \eqref{co-differential operator} we have
\begin{equation*}
\begin{aligned}
(\delta_{\nabla}\Omega)_{ij}(e_{k}) &=(\delta_{\nabla}\Omega_{ij})(e_{k}) \\
&= -\sum_{l} (\nabla_{e_{l}}\Omega_{ij})(e_{l}, e_{k}) \\
&= -\sum_{l} e_{l} \left( \Omega_{ij}(e_{l}, e_{k}) \right) \\
&= -\sum_{l}e_l\Omega_{ijlk},
\end{aligned}
\end{equation*}
where $\Omega_{ijlk}$ represents the components of the Riemannian curvature tensor of $M$.
Applying the second Bianchi identity yields
\begin{equation*}
    \sum_{l} e_l \Omega_{ijlk} = -\sum_{l} \left(e_i \Omega_{jllk} + e_j \Omega_{lilk}\right) = \nabla_{i}\Ric_{jk} - \nabla_{j}\Ric_{ik}.
\end{equation*}
Combining these relations, we obtain
\begin{equation*}
    (\delta_{\nabla}\Omega)_{ij}(e_{k}) = -\nabla_{i}\Ric_{jk} + \nabla_{j}\Ric_{ik}.
\end{equation*}
Consequently, the Levi-Civita connection is Yang-Mills if and only if the Ricci
tensor is Codazzi.
\end{proof}

\subsection{Proof of Theorem \ref{thm-1}}
\begin{proof}  
We calculate the variation of the normal curvature functional by splitting it into the variation of the norm of the normal curvature and the variation of the volume:
\begin{equation*}
    \frac{d}{dt}\bigg|_{t=0} \int_{M} \|\Omega^{\bot}\|^{2} \, dvol = \int_{M} \frac{\partial}{\partial t}(\|\Omega^{\bot}\|^{2})\bigg|_{t=0} \, dvol + \int_{M} \|\Omega^{\bot}\|^{2} \frac{\partial}{\partial t}(dvol)\bigg|_{t=0}.
\end{equation*}

For the volume component, applying Chern's first-order variational formula for volume functional (cf. \cite{C-S-S}) and $a_i = 0$, we have
\begin{equation}\label{vol}
\begin{aligned}
    \int_{M} \|\Omega^{\bot}\|^{2} \frac{\partial}{\partial t}(dvol)\bigg|_{t=0} &= -\int_{M} \|\Omega^{\bot}\|^{2} \left( \sum_{\alpha,i} a_{\alpha}h^{\alpha}_{ii} \right) dvol \\
    &= -\int_{M} \|\Omega^{\bot}\|^{2} \left\langle \sum_{\alpha} a_{\alpha}e_{\alpha}, H \right\rangle dvol.
\end{aligned}
\end{equation}

For the variation of the norm of the normal curvature, using the inner product definition, we obtain
\begin{equation*}
    \frac{\partial}{\partial t}(\|\Omega^{\bot}\|^{2}) = 2 \left\langle \frac{\partial}{\partial t}\Omega^{\bot}, \Omega^{\bot} \right\rangle + 2\sum_{\alpha,\beta,i,j} \Omega^{\bot}_{\alpha\beta}(e_{i}, e_{j})\Omega^{\bot}_{\alpha\beta}\left(\frac{\partial}{\partial t}e_{i}, e_{j}\right).
\end{equation*}

By Lemma \ref{var of curvature forms}, it follows that
\begin{align*}
2\left\langle \frac{\partial}{\partial t}\Omega^{\bot}, \Omega^{\bot} \right\rangle &= -2\langle d_{\nabla^{\bot}}\tau^{\bot}, \Omega^{\bot}\rangle+2\langle [b^{\bot},\Omega^{\bot}], \Omega^{\bot}\rangle\\
&=-2\langle d_{\nabla^{\bot}}\tau^{\bot}, \Omega^{\bot}\rangle-2\sum\limits_{i<j}\trace(b^{\bot}\Omega^{\bot}(e_i,e_j)\Omega^{\bot}(e_i,e_j)-\Omega^{\bot}(e_i,e_j)b^{\bot}\Omega^{\bot}(e_i,e_j))\\
&=-2\langle d_{\nabla^{\bot}}\tau^{\bot}, \Omega^{\bot}\rangle.
\end{align*}

Integrating by parts over $M$ leads to
\begin{equation*}
    -2\int_{M} \langle d_{\nabla^{\bot}}\tau^{\bot}, \Omega^{\bot} \rangle \, dvol = -2\int_{M} \langle \tau^{\bot}, \delta_{\nabla^{\bot}}\Omega^{\bot} \rangle \, dvol.
\end{equation*}

By Lemma \ref{from structural equations}(2), $b_{i\alpha} = \nabla_{e_i}^{\bot}a_{\alpha}$. Hence,
\begin{equation}\label{tau-omega}
\begin{aligned}
    -2\int_{M} \langle \tau^{\bot}, \delta_{\nabla^{\bot}}\Omega^{\bot} \rangle \, dvol &= -2\int_{M} \sum_{\alpha, \beta, k} \tau^{\bot}_{\alpha\beta}(e_k) (\delta_{\nabla^{\bot}}\Omega^{\bot})_{\alpha\beta}(e_{k}) \, dvol \\
    &= 2\int_{M} \sum_{\alpha, \beta,i, k} (b_{i\alpha}\theta_{i\beta}-b_{i\beta}\theta_{i\alpha})(e_k) \, (\delta_{\nabla^{\bot}}\Omega^{\bot})_{\alpha\beta}(e_{k}) \, dvol \\
    &= 4\int_{M} \sum_{\alpha, \beta,i, k} b_{i\alpha}h^{\beta}_{ik}  (\delta_{\nabla^{\bot}}\Omega^{\bot})_{\alpha\beta}(e_{k}) \, dvol \\
    &= 4\int_{M} \sum_{\alpha, \beta,i, k} (\nabla_{e_{i}}^{\bot}a_{\alpha})h^{\beta}_{ik}  (\delta_{\nabla^{\bot}}\Omega^{\bot})_{\alpha\beta}(e_{k}) \, dvol \\
    &= -4\int_{M} \sum_{\alpha, \beta,i, k} a_{\alpha} D_{e_i}\left(h^{\beta}_{ik}  (\delta_{\nabla^{\bot}}\Omega^{\bot})_{\alpha\beta}(e_{k})\right) \, dvol.
\end{aligned}
\end{equation}

To calculate the second term of $\frac{\partial}{\partial t}(\|\Omega^{\bot}\|^{2})$, we first consider $\frac{\partial}{\partial t}e_{i}$. We know from Lemma \ref{from structural equations}(1) that
\[\frac{\partial}{\partial t}\bigg|_{t=0}\theta_i=\sum_j b_{ij}\theta_{j}-\sum_\alpha a_{\alpha}\theta_{i\alpha}=\sum_{j}\left(b_{ij}-\sum_{\alpha}a_{\alpha}h_{ij}^{\alpha}\right)\theta_{j}.\]
It thus can be deduced from the identity $0=\frac{\partial}{\partial t}(\theta_{i}(e_{j}))=\left(\frac{\partial}{\partial t}\theta_{i}\right)(e_{j})+\theta_{i}\left(\frac{\partial}{\partial t}e_{j}\right)$ that
\begin{equation}\label{partial t e_i}
\frac{\partial}{\partial t}\bigg|_{t=0}e_{i}=\sum_{j}\left(b_{ij}+\sum_{\alpha}a_{\alpha}h_{ij}^{\alpha}\right)e_{j}=:\sum_{j}\lambda_{ij}e_{j}.
\end{equation}

By contracting the normal curvature tensors via the Ricci equation, we find
\begin{equation}\label{algebraic}
\begin{aligned}
    2\sum_{\beta,\gamma,i,j} \Omega^{\bot}_{\beta\gamma}(e_{i}, e_{j})\Omega^{\bot}_{\beta\gamma}\left(\frac{\partial}{\partial t}e_{i}, e_{j}\right) &= 2\sum_{\beta,\gamma,i,j,k} \lambda_{ik}\Omega^{\bot}_{\beta\gamma}(e_{i}, e_{j})\Omega^{\bot}_{\beta\gamma}(e_{k}, e_{j}) \\
	&= 2\sum_{\beta,\gamma,i,j,k}\frac{\lambda_{ik}+\lambda_{ki}}{2}\,\Omega^{\bot}_{\beta\gamma}(e_{i}, e_{j})\Omega^{\bot}_{\beta\gamma}(e_{k}, e_{j}) \\
    &= 2\sum_{\alpha,\beta,\gamma,i,j,k} a_{\alpha}h^{\alpha}_{ki}\Omega^{\bot}_{\beta\gamma}(e_{i}, e_{j})\Omega^{\bot}_{\beta\gamma}(e_{k}, e_{j}) \\
    &= -2\sum_{\alpha,\beta,\gamma,i,j,k} a_{\alpha}h^{\alpha}_{ki}[h^{\beta}, h^{\gamma}]_{ij}[h^{\beta}, h^{\gamma}]_{jk} \\
    &= -4\sum_{\alpha,\beta,\gamma} a_{\alpha}\trace(h^{\alpha}h^{\beta}h^{\gamma}h^{\beta}h^{\gamma} - h^{\alpha}h^{\beta}h^{\gamma}h^{\gamma}h^{\beta}).
\end{aligned}
\end{equation}

Combining   \eqref{vol},   \eqref{tau-omega}, and   \eqref{algebraic}, we arrive at
\begin{equation*}
\begin{aligned}
    \frac{d}{d t}\bigg|_{t=0}\int_{M} \|\Omega^{\bot}\|^{2} \, dvol &= \int_{M} \sum_{\alpha} -4a_\alpha \Bigg( \sum_{\beta,i,k} D_{e_i}\left(h^{\beta}_{ik}  (\delta_{\nabla^{\bot}}\Omega^{\bot})_{\alpha\beta}(e_{k})\right) \\
    &\quad + \sum_{\beta,\gamma} \trace(h^{\alpha}h^{\beta}h^{\gamma}h^{\beta}h^{\gamma} - h^{\alpha}h^{\beta}h^{\gamma}h^{\gamma}h^{\beta}) \Bigg) \, dvol \\
    &\quad - \int_{M} \|\Omega^{\bot}\|^{2} \left\langle \sum_{\alpha} a_{\alpha} e_{\alpha}, H \right\rangle \, dvol \\
    &= -4\int_{M} \left\langle \sum_{\alpha} a_{\alpha} e_{\alpha}, \Hn \right\rangle \, dvol - \int_{M} \|\Omega^{\bot}\|^{2} \left\langle \sum_{\alpha} a_{\alpha} e_{\alpha}, H \right\rangle \, dvol \\
    &= \int_{M} \left\langle \frac{\partial F}{\partial t}\bigg|_{t=0}, -4\Hn - \|\Omega^{\bot}\|^{2}H \right\rangle \, dvol.
\end{aligned}
\end{equation*}

Thus, $M$ is a Normal-Yang-Mills submanifold if and only if 
\begin{equation*}
4\Hn + \|\Omega^{\bot}\|^{2}H = 0.
\end{equation*}

In particular, if $M$ is a minimal submanifold, this condition reduces to $\Hn=0$.
\end{proof}

\subsection{Proof of Theorem \ref{thm-2}}
\begin{proof}
We calculate the variation of the curvature functional by splitting it into the variation of the norm of the curvature and the variation of the volume:
\begin{equation*}
    \frac{d}{d t}\bigg|_{t=0}\int_{M} \|\Omega\|^{2} dvol = \int_{M} \frac{\partial}{\partial t}(\|\Omega\|^{2})\bigg|_{t=0} dvol + \int_{M} \|\Omega\|^{2} \frac{\partial}{\partial t}(dvol)\bigg|_{t=0}.
\end{equation*}

Similar to \eqref{vol}, the variation of the volume yields
\begin{equation}\label{vol-tangent}
    \int_{M} \|\Omega\|^{2} \frac{\partial}{\partial t}(dvol)\bigg|_{t=0} = -\int_{M} \|\Omega\|^{2} \left\langle \sum_{\alpha} a_{\alpha}e_{\alpha}, H \right\rangle dvol.
\end{equation}

For the variation of the norm of the curvature, we have
\begin{equation*}
    \frac{\partial}{\partial t}(\|\Omega\|^{2}) = 2\left\langle \frac{\partial}{\partial t}\Omega, \Omega \right\rangle + 2\sum_{i,j,p,q} \Omega_{ij}(e_{p},e_{q})\Omega_{ij}\left(\frac{\partial}{\partial t}e_{p},e_{q}\right).
\end{equation*}

By Lemma \ref{var of curvature forms}, we have $\frac{\partial}{\partial t}\Omega = -d_{\nabla}\tau + [b,\Omega]$. Since the inner product of the commutator $[b, \Omega]$ with $\Omega$ vanishes, the first term reduces to
\begin{equation}\label{dtheta-tau-tangent}
    2\left\langle \frac{\partial}{\partial t}\Omega, \Omega \right\rangle = -2\langle d_{\nabla}\tau, \Omega \rangle.
\end{equation}

Integrating \eqref{dtheta-tau-tangent} over $M$ by parts, we obtain
\begin{equation*}
    -2\int_{M} \langle d_{\nabla}\tau, \Omega \rangle dvol = -2\int_{M} \langle \tau, \delta_{\nabla}\Omega \rangle dvol.
\end{equation*}

By the variation identity from Lemma \ref{from structural equations}(2), $b_{i\gamma} = \nabla_{e_i}^{\bot}a_\gamma$. Thus
\begin{equation}\label{tau-omega-tangent}
\begin{aligned}
    -2\int_{M} \langle \tau, \delta_{\nabla}\Omega \rangle dvol &= -2\int_{M} \sum_{i,j,k} \tau_{ij}(e_k) (\delta_{\nabla}\Omega)_{ij}(e_{k}) dvol \\
    &= 4\int_{M} \sum_{\gamma,i, j,k} b_{i\gamma} h^{\gamma}_{jk} (\delta_{\nabla}\Omega)_{ij}(e_{k}) \, dvol \\
    &= 4\int_{M} \sum_{\gamma,i, j,k} (\nabla_{e_i}^{\bot}a_{\gamma})h^{\gamma}_{jk} (\delta_{\nabla}\Omega)_{ij}(e_{k}) \, dvol \\
    &= -4\int_{M} \sum_{\gamma,i, j,k} a_{\gamma}D_i\left(h^{\gamma}_{jk} (\delta_{\nabla}\Omega)_{ij}(e_{k})\right) dvol.
\end{aligned}
\end{equation}

By the Codazzi equation and Lemma \ref{codiff-riem-curv}, we get
\begin{equation*}
\begin{aligned}
    \sum_{i,j,k}D_i\left(h^{\gamma}_{jk} (\delta_{\nabla}\Omega)_{ij}(e_{k})\right) &= \sum_{i,j,k} (Dh)^\gamma_{ijk} (\delta_{\nabla}\Omega)_{ij}(e_{k})+ \sum_{i,j,k} h^\gamma_{jk}\nabla_i (\delta_{\nabla}\Omega)_{ij}(e_{k})\\
    &= \sum_{i,j,k} h^\gamma_{jk}\nabla_i (-\nabla_{i}\Ric_{jk} + \nabla_{j}\Ric_{ik}).
\end{aligned}
\end{equation*}

For the second term of $\frac{\partial}{\partial t}(\|\Omega\|^{2})$, using the Gauss equation \eqref{Gauss-Codazzi-Ricci} and   \eqref{partial t e_i}, we obtain 
\begin{equation}\label{algebraic-tangent}
\begin{aligned}
2\sum_{i,j,k,l} \Omega_{ij}(e_{k},e_{l})\Omega_{ij}\left(\frac{\partial}{\partial t}e_{k},e_{l}\right)&=2\sum\limits_{\alpha,i,j,k,l,m} a_{\alpha}h_{km}^{\alpha}\Omega_{ij}(e_{k},e_{l})\Omega_{ij}(e_{m},e_{l}) \\ 
&= 4(n-1)\sum\limits_{\alpha}a_{\alpha} H^\alpha + 8\sum\limits_{\alpha, \beta}a_{\alpha} H^\beta \trace(h^\alpha h^\beta)\\ 
&\quad - 8\sum\limits_{\alpha,\beta} a_{\alpha}\trace(h^{\alpha}h^{\beta}h^{\beta}) \\ 
&\quad + 4\sum\limits_{\alpha,\beta,\gamma} a_{\alpha}\trace(h^{\alpha}h^{\beta}h^{\gamma})\trace(h^{\beta}h^{\gamma}) \\ 
&\quad - 4\sum\limits_{\alpha,\beta,\gamma} a_{\alpha}\trace(h^{\alpha}h^{\beta}h^{\gamma}h^{\beta}h^{\gamma}).
\end{aligned}
\end{equation}

Combining   \eqref{vol-tangent},   \eqref{tau-omega-tangent}, and   \eqref{algebraic-tangent}, we derive

\begin{align*}
    \frac{d}{d t}\bigg|_{t=0}\int_{M} \|\Omega\|^{2} dvol
    &= \int_{M} \sum_{\alpha} a_\alpha \Bigg( 4\sum_{i,j,k} h^\alpha_{jk}\nabla_i (\nabla_{i}\Ric_{jk} - \nabla_{j}\Ric_{ik}) \\
    &\quad + 4(n-1)H^\alpha + 8\sum\limits_{\beta}H^\beta \trace(h^\alpha h^\beta)- 8\sum_{\beta}\trace(h^{\alpha}h^{\beta}h^{\beta}) \\
    &\quad + 4\sum_{\beta,\gamma}\trace(h^{\alpha}h^{\beta}h^{\gamma})\trace(h^{\beta}h^{\gamma})- 4\sum_{\beta,\gamma}\trace(h^{\alpha}h^{\beta}h^{\gamma}h^{\beta}h^{\gamma}) \Bigg) dvol \\
    &\quad - \int_{M} \|\Omega\|^{2} \left\langle \sum_{\alpha} a_{\alpha}e_{\alpha}, H \right\rangle dvol \\
    &= \int_{M} \left\langle \sum_{\alpha} a_{\alpha}e_{\alpha}, 4\Ht \right\rangle dvol + \int_{M} \left\langle \sum_{\alpha} a_{\alpha}e_{\alpha}, \Bigg(4(n-1) - \|\Omega\|^{2}\Bigg)H \right\rangle dvol     \\
    &\quad + \int_{M} \left\langle \sum_{\alpha} a_{\alpha}e_{\alpha},  8\sum_{\alpha,\beta} H^\beta \trace(h^\alpha h^\beta) e_\alpha \right\rangle dvol\\
    &= \int_{M} \left\langle \frac{\partial F}{\partial t}\bigg|_{t=0}, 4\Ht + \Big(4(n-1) - \|\Omega\|^{2}\Big)H+8\sum_{\alpha,\beta} H^\beta \trace(h^\alpha h^\beta) e_\alpha \right\rangle dvol.
\end{align*}

Therefore, $M$ is a Tangent-Yang-Mills submanifold if and only if it satisfies the Euler-Lagrange equation

\begin{equation*}
4\Ht + \Big(4(n-1) - \|\Omega\|^{2}\Big)H+8\sum_{\alpha,\beta} H^\beta \trace(h^\alpha h^\beta) e_\alpha  = 0.
\end{equation*}

In particular, if $M$ is a minimal submanifold, this condition reduces to $\Ht=0$.
\end{proof}

\section{Examples of Normal and Tangent Yang-Mills Submanifolds}

This section discusses non-trivial examples of Normal-Yang-Mills and Tangent-Yang-Mills submanifolds whose respective curvature tensors do not satisfy the classical Yang-Mills equations.

\subsection{OT-FKM Isoparametric Hypersurfaces and Focal Submanifolds}

We first briefly review the algebraic foundations of symmetric Clifford systems. A symmetric Clifford system $(P_{0},\dots,P_{m})$ on $\mathbb{R}^{2l}$ is an $(m+1)$-tuple of orthogonal symmetric matrices that satisfy $P_{i}P_{j}+P_{j}P_{i}=2\delta_{ij}I_{2l}$, where $l = k\delta(m)>m+1$ for some positive integer $k$, and $\delta(m)$ is the dimension of the irreducible module of the Clifford algebra $\mathcal{Cl}_{m-1}$. The dimensions $\delta(m)$ satisfy the relation $\delta(m+8)=16\delta(m)$ and are given for small values of $m$ in Table \ref{table-1} below:
\begin{table}[h]
  \centering
  \caption{Irreducible dimension for Clifford algebra $\mathcal{Cl}_{m-1}$} \label{table-1}
  \begin{tabular}{|c|c|c|c|c|c|c|c|c|c|}
  \hline
  $m$&1&2&3&4&5&6&7&8&$\cdots m+8$\\
  \hline
  $\delta(m)$&1&2&4&4&8&8&8&8&$\cdots 16\delta(m)$\\
  \hline
  \end{tabular}
\end{table}

For a given symmetric Clifford system on $\mathbb{R}^{2l}$, we can choose a set of orthogonal matrices $\{ E_{1}, E_{2}, \cdots, E_{m-1} \}$ on $\mathbb{R}^{l}$ with the Euclidean metric, which satisfy $E_{\alpha}E_{\beta}+E_{\beta}E_{\alpha}=-2\delta_{\alpha\beta}I_l$ for $1\leqslant\alpha,\beta\leqslant m-1$ and
\begin{equation}\label{Clifford-sys-alg}
\begin{aligned}
    &P_{0}=\begin{pmatrix}
        I_l & 0 \\
        0  & -I_l
    \end{pmatrix},
    P_{1}=\begin{pmatrix}
        0 & I_l \\
       I_l & 0
    \end{pmatrix}, \\
    &P_{\alpha}=\begin{pmatrix}
        0 & E_{\alpha-1} \\
        -E_{\alpha-1}  & 0
    \end{pmatrix}, \ \mathrm{for} \ 2\leqslant\alpha\leqslant m.
\end{aligned}
\end{equation}

As proved in \cite{F-K-M}, the algebraic structure of these systems depends on $m$. When $m\not\equiv 0 \pmod{4}$, there exists exactly one kind of OT-FKM type isoparametric family for a given $l=k\delta(m)$. When $m\equiv 0 \pmod{4}$, there are distinct kinds of OT-FKM type isoparametric families distinguished by $\left|\trace(P_{0}\cdots P_{m})\right|$. Specifically, the family with $P_{0}P_{1}\cdots P_{m}=\pm I_{2l}$, where we take the $+$ sign without loss of generality, is called the definite family, and the others with $P_{0}P_{1}\cdots P_{m}\neq \pm I_{2l}$ are called indefinite.

Based on the Clifford system, we define the following homogeneous polynomial on $\mathbb{R}^{2l}$ of degree 4 (see \cite{F-K-M}):
$$\tilde{f}(x)=|x|^{4}-2\sum_{i=0}^{m}\langle P_{i}x, x \rangle^{2}.$$
The restriction to the unit sphere $f=\tilde{f}|_{S^{2l-1}}$ has $\mathrm{Im}f=[-1,1]$. An isoparametric hypersurface $M^{2l-2}$ of OT-FKM type is any regular level set of $f$, that is, $M=f^{-1}(\cos 4\theta)$ for some constant $0<\theta<\tfrac{\pi}{4}$ (See the classifications and developments of isoparametric hypersurfaces in the excellent book \cite{CR2015} and recent surveys \cite{Chi18, GQTY2025} and references therein).  It has four distinct constant principal curvatures $\lambda_{k}=\cot(\theta+\tfrac{k-1}{4}\pi)$ with multiplicities $m_{1},m_{2},m_{1},m_{2}$ respectively, where $m_{1}=m$ and $m_{2}=l-m-1$. A table of small multiplicities $(m_1, m_2)$ is provided below, where the number of underlines indicates the number of distinct indefinite families for $m\equiv 0 \pmod{4}$.

\begin{table}[h]\centering\footnotesize 
\caption{Multiplicities $(m,l-m-1)$ for OT-FKM isoparametric hypersurfaces} \label{table-2}
\setlength{\tabcolsep}{0.8pt} 
\renewcommand{\arraystretch}{1.2} 
\begin{tabular}{|c|ccccccccccccc|}\hline$k \setminus \delta(m)$ & 1 & 2 & 4 & 4 & 8 & 8 & 8 & 8 & 16 & 32 & 64 & 64 & $\dots$ \\ \hline1 & $-$ & $-$ & $-$ & $-$ & (5, 2) & (6, 1) & $-$ & $-$ & (9, 6) & (10, 21) & (11, 52) & (12, 51) & $\dots$ \\2 & $-$ & (2, 1) & (3, 4) & \underline{(4, 3)} & (5, 10) & (6, 9) & (7, 8) & \underline{(8, 7)} & (9, 22) & (10, 53) & (11, 116) & \underline{(12, 115)} & $\dots$ \\3 & (1, 1) & (2, 3) & (3, 8) & \underline{(4, 7)} & . & . & . & \underline{(8, 15)}& . & . & . & \underline{(12, 179)} & $\dots$ \\4 & (1, 2) & (2, 5) & (3, 12) & \underline{\underline{(4, 11)}} & . & . & . & \underline{\underline{(8, 23)}} & . & . & . & \underline{\underline{(12, 243)}} & $\dots$ \\5 & (1, 3) & (2, 7) & (3, 16) & \underline{\underline{(4, 15)}} & . & . & . & \underline{\underline{(8, 31)}} & . & . & . & \underline{\underline{(12, 307)}} & $\dots$ \\$\vdots$ & $\vdots$ & $\vdots$ & $\vdots$ & $\vdots$ & $\vdots$ & $\vdots$ & $\vdots$ & $\vdots$ & $\vdots$ & $\vdots$ & $\vdots$ & $\vdots$ & $\ddots$ \\ \hline
\end{tabular}\end{table}

The two singular level sets $M_{\pm}=f^{-1}(\pm1)$ are called focal submanifolds. In particular, the focal submanifold $M_+$ of this OT-FKM type is given by
\begin{equation}\label{def-M+}
\begin{aligned}
    M_{+}=\{ x\in S^{2l-1} \mid \langle P_{0}x, x \rangle=\cdots=\langle P_{m}x, x \rangle=0 \},
\end{aligned}
\end{equation}
where the normal space of $M_{+}$ at $x$ is equal to $\mathrm{Span}\{P_{0}x, P_{1}x, \cdots, P_{m}x\}$.

\subsection{Local Formulas on $M_+$}

Let $x \in M_+$. For any vectors $X, Y$ in the Euclidean space $\mathbb{R}^{2l}$, we denote their tangent components on $T_x M_+$ by $X^{\top}, Y^{\top}$, and their normal components in $\mathbb{R}^{2l}$ by $X^{\bot}, Y^{\bot}$. Their inner product can be orthogonally decomposed as
\begin{equation}\label{Vector decomposition}
\langle X, Y \rangle=\langle X^{\top}, Y^{\top} \rangle+\langle X^{\bot}, Y^{\bot} \rangle.
\end{equation}

Taking a local orthonormal basis $\{e_k\}$ of $T_x M_+$, the inner product of the tangent components can be computed by
\begin{equation}\label{Calculation of tangent components}
\langle X^{\top}, Y^{\top} \rangle=\sum\limits_{k}\langle X, e_k \rangle\langle Y, e_k \rangle.
\end{equation}

Since the normal space of $M_+$ in $\mathbb{R}^{2l}$ is orthonormally spanned by $x$ and $\{P_\gamma x\}_{\gamma=0}^m$, the inner product of the normal components can be computed as follows
\begin{equation}\label{Calculation of normal components}
\langle X^{\bot}, Y^{\bot} \rangle=\langle X, x \rangle\langle Y, x \rangle+\sum\limits_{\gamma}\langle X, P_{\gamma}x \rangle\langle Y, P_{\gamma}x \rangle.
\end{equation}

The shape operator corresponding to the normal vector $P_{\alpha}x$ can be represented by $A_{\alpha}(v)=-(P_{\alpha}v)^{\top}$ for any tangent vector $v$. Thus, the second fundamental form satisfies
\begin{equation}\label{second fundamental form}
h_{ij}^{\alpha}=\langle A_{\alpha}(e_{i}), e_{j} \rangle =-\langle (P_{\alpha}e_{i})^{\top}, e_{j} \rangle =-\langle P_{\alpha}e_{i}, e_{j} \rangle.
\end{equation}

In particular, the minimality of $M_{+}$ is equivalent to the condition:
\begin{equation*}
\forall \alpha,\ \sum\limits_{i}h^{\alpha}_{ii}=-\sum\limits_{i}\langle P_{\alpha}e_{i}, e_{i} \rangle=0.
\end{equation*}

From the anti-commutativity of the Clifford system, we have $\langle P_{\alpha}P_{\beta}x,x\rangle=\delta_{\alpha\beta}$ for any $x$ on the sphere. Furthermore, for any $x \in M_+$, the following identity holds
\begin{equation}\label{P_alpha P_beta P_gamma}
\langle P_{\alpha}P_{\beta}P_{\gamma}x, x \rangle=0
\end{equation}
for any choice of $\alpha,\beta,\gamma\in\{0,\dots,m\}$.

For $\alpha,\beta,\gamma\in\{0,\dots,m\}$, we define two fundamental tensors on $M_+$. For any $v\in T_{x}M_{+}$, let
\begin{align*}
g_{\alpha\beta}(v)&:=\langle P_{\alpha}P_{\beta}x, v \rangle,\\
g_{\alpha\beta\gamma}(v)&:=\langle P_{\alpha}P_{\beta}P_{\gamma}x, v \rangle.
\end{align*}

\begin{lem}\label{h alpha g beta gamma}
The tensors $g_{\alpha\beta}$ and $g_{\alpha\beta\gamma}$ satisfy the following properties:
\begin{enumerate}
\item $g_{\alpha\beta}$ and $g_{\alpha\beta\gamma}$ are anti-symmetric in their respective indices.
\item $\sum\limits_{j}h^{\alpha}_{ij}g_{\beta\gamma}(e_{j})=-g_{\alpha\beta\gamma}(e_{i})$.
\item $\sum\limits_{j}h^{\alpha}_{ij}g_{\beta\gamma\delta}(e_{j})=-\langle P_{\alpha}P_{\beta}P_{\gamma}P_{\delta}x, e_{i} \rangle+\sum\limits_{\varepsilon}g_{\alpha\varepsilon}(e_{i})\langle P_{\varepsilon}P_{\beta}P_{\gamma}P_{\delta}x, x \rangle$.
\item $\sum\limits_{i} g_{\alpha\beta}(e_{i})g_{\gamma\delta\varepsilon}(e_{i})=-\langle P_{\alpha}P_{\beta}P_{\gamma}P_{\delta}P_{\varepsilon}x, x \rangle$.
\end{enumerate}
\end{lem}
\begin{proof}
(1) For any tangent vector $v$ of $M_+$, it is clear that $g_{\alpha\alpha}(v)=\langle x,v\rangle=0$. Assume $\alpha\neq \beta$, then $g_{\alpha\beta}=-g_{\beta\alpha}$ follows immediately from the anti-commutativity of the Clifford system. Since $P_{\gamma}x$ is a normal vector of $M_{+}$ at $x$, we have
\begin{equation*}
g_{\alpha\beta\gamma}(v)=\langle P_{\alpha}P_{\beta}P_{\gamma}x, v \rangle=
\begin{cases}
\langle P_{\gamma}x, v \rangle  & \alpha=\beta\\
-\langle P_{\beta}P_{\alpha}P_{\gamma}x, v \rangle & \alpha\neq\beta
\end{cases}
=-g_{\beta\alpha\gamma}(v).
\end{equation*}
Similarly, $g_{\alpha\beta\gamma}=-g_{\alpha\gamma\beta}$, which implies $g_{\alpha\beta\gamma}$ is anti-symmetric with respect to $\alpha,\beta,\gamma$.

(2) Using \eqref{Vector decomposition}, \eqref{Calculation of tangent components} and \eqref{Calculation of normal components}, we obtain
\begin{align*}
\sum_{j}h^{\alpha}_{ij}g_{\beta\gamma}(e_{j})&=-\sum_{j}\langle P_{\alpha}e_{i},e_{j}\rangle \langle P_{\beta}P_{\gamma}x,e_{j}\rangle\\
&=-\langle (P_{\alpha}e_{i})^{\top}, (P_{\beta}P_{\gamma}x)^{\top}\rangle\\
&=-\left(\langle P_{\alpha}e_{i}, P_{\beta}P_{\gamma}x \rangle-\langle P_{\alpha}e_{i}, x \rangle \langle P_{\beta}P_{\gamma}x,x \rangle-\sum_{\delta}\langle P_{\alpha}e_{i}, P_{\delta}x \rangle \langle P_{\beta}P_{\gamma}x,P_{\delta}x \rangle\right).
\end{align*}

By   \eqref{P_alpha P_beta P_gamma}, $\langle P_{\beta}P_{\gamma}x, P_{\delta}x \rangle$ vanishes for each $\delta$. It follows that$$\sum_{j}h^{\alpha}_{ij}g_{\beta\gamma}(e_{j}) = -\langle P_{\alpha}e_{i}, P_{\beta}P_{\gamma}x \rangle = -g_{\alpha\beta\gamma}(e_{i}).$$

(3) Applying the same decomposition to $\sum\limits_{j}h^{\alpha}_{ij}g_{\beta\gamma\delta}(e_{j})$, we can obtain
\begin{align*}
\sum_{j}h^{\alpha}_{ij}g_{\beta\gamma\delta}(e_{j}) &= -\langle (P_{\alpha}e_{i})^{\top}, (P_{\beta}P_{\gamma}P_{\delta}x)^{\top}\rangle \\
&= -\langle P_{\alpha}e_{i}, P_{\beta}P_{\gamma}P_{\delta}x \rangle + \sum_{\varepsilon}\langle P_{\alpha}e_{i}, P_{\varepsilon}x\rangle\langle P_{\beta}P_{\gamma}P_{\delta}x, P_{\varepsilon}x\rangle \\
&= -\langle P_{\alpha}P_{\beta}P_{\gamma}P_{\delta}x, e_{i} \rangle + \sum_{\varepsilon}g_{\alpha\varepsilon}(e_{i})\langle P_{\varepsilon}P_{\beta}P_{\gamma}P_{\delta}x, x\rangle.
\end{align*}

(4) Summing over the orthonormal basis $\{e_i\}$ of $T_x M_+$, we have
\begin{align*}
\sum_{i} g_{\alpha\beta}(e_{i})g_{\gamma\delta\varepsilon}(e_{i}) &= \langle (P_{\alpha}P_{\beta}x)^{\top}, (P_{\gamma}P_{\delta}P_{\varepsilon}x)^{\top} \rangle \\
&= \langle P_{\alpha}P_{\beta}x, P_{\gamma}P_{\delta}P_{\varepsilon}x \rangle - \langle P_{\alpha}P_{\beta}x, x \rangle\langle P_{\gamma}P_{\delta}P_{\varepsilon}x, x \rangle \\
&\quad - \sum_{\mu}\langle P_{\alpha}P_{\beta}x, P_{\mu}x \rangle\langle P_{\gamma}P_{\delta}P_{\varepsilon}x, P_{\mu}x \rangle \\
&= -\langle P_{\alpha}P_{\beta}P_{\gamma}P_{\delta}P_{\varepsilon}x, x \rangle.
\end{align*}
\end{proof}

\begin{lem}\label{h alpha h beta}
The components of the second fundamental form satisfy the relation:
\begin{equation*}
(h^{\alpha}h^{\beta})_{ij}=\langle P_{\alpha}e_{i}, P_{\beta}e_{j} \rangle-\sum\limits_{\gamma\neq\alpha,\beta}g_{\alpha\gamma}(e_{i})g_{\beta\gamma}(e_{j}).
\end{equation*}
In particular, 
\begin{equation*}
\sum\limits_{\alpha}(h^{\alpha})^{2}_{ij}=(m+1)\delta_{ij}-\sum\limits_{\alpha}\sum\limits_{\gamma\neq\alpha}g_{\alpha\gamma}(e_{i})g_{\alpha\gamma}(e_{j}).
\end{equation*}
\end{lem}
\begin{proof}
From   \eqref{Calculation of normal components}, it follows that
\begin{equation*}\begin{aligned}
\langle (P_{\alpha}e_{i})^{\bot}, (P_{\beta}e_{j})^{\bot} \rangle&=\langle P_{\alpha}e_{i}, x \rangle\langle P_{\beta}e_{j}, x \rangle+\sum\limits_{\gamma\neq\alpha,\beta}\langle P_{\alpha}e_{i}, P_{\gamma}x \rangle\langle P_{\beta}e_{j}, P_{\gamma}x \rangle \\
&=\sum\limits_{\gamma\neq\alpha,\beta}g_{\alpha\gamma}(e_{i})g_{\beta\gamma}(e_{j}).
\end{aligned}
\end{equation*}
Combining this with   \eqref{Vector decomposition} and   \eqref{second fundamental form}, we compute
\begin{align*}
(h^{\alpha}h^{\beta})_{ij}&=\sum_{k}h_{ik}^{\alpha}h_{kj}^{\beta} \\
&=\sum_{k}\langle P_{\alpha}e_{i}, e_{k} \rangle\langle P_{\beta}e_{k}, e_{j} \rangle \\
&=\langle (P_{\alpha}e_{i})^{\top}, (P_{\beta}e_{j})^{\top} \rangle\\
&=\langle P_{\alpha}e_{i}, P_{\beta}e_{j} \rangle-\langle (P_{\alpha}e_{i})^{\bot}, (P_{\beta}e_{j})^{\bot} \rangle \\
&=\langle P_{\alpha}e_{i}, P_{\beta}e_{j} \rangle-\sum\limits_{\gamma\neq\alpha,\beta}g_{\alpha\gamma}(e_{i})g_{\beta\gamma}(e_{j}).
\end{align*}

Letting $\beta=\alpha$, we sum over $\alpha$ to obtain
\begin{align*}
\sum\limits_{\alpha}(h^{\alpha})^{2}_{ij}&=\sum\limits_{\alpha}\langle P_{\alpha}e_{i}, P_{\alpha}e_{j} \rangle-\sum\limits_{\alpha}\sum\limits_{\gamma\neq\alpha}g_{\alpha\gamma}(e_{i})g_{\alpha\gamma}(e_{j}) \\
&=(m+1)\delta_{ij}-\sum\limits_{\alpha}\sum\limits_{\gamma\neq\alpha}g_{\alpha\gamma}(e_{i})g_{\alpha\gamma}(e_{j}).
\end{align*}
\end{proof}

To simplify notation, we introduce a symmetric $(0,2)$-tensor $G$ defined by 
\begin{equation*}
G_{ij}=\sum\limits_{\gamma}\sum\limits_{\delta\neq\gamma}g_{\gamma\delta}(e_{i})g_{\gamma\delta}(e_{j}).
\end{equation*}
With this definition, the second identity in Lemma \ref{h alpha h beta} can be rewritten as
\begin{equation*}
\sum\limits_{\alpha}(h^{\alpha})^{2}_{ij}=(m+1)\delta_{ij}-G_{ij}.
\end{equation*}

\begin{lem}\label{trace degree 5}
On the focal submanifold $M_{+}$, the second fundamental form satisfies the following trace identities:
\begin{align*}
\trace\sum_{\beta,\gamma}(h^{\alpha}h^{\beta}h^{\gamma}h^{\gamma}h^{\beta})&=2\sum\limits_{\beta\neq\alpha}\sum\limits_{\gamma\neq\alpha,\beta}\sum\limits_{\delta\neq\alpha,\beta,\gamma}\sum\limits_{\varepsilon\neq\alpha,\beta,\gamma,\delta}\langle P_{\beta}P_{\gamma}P_{\delta}P_{\varepsilon}x, x\rangle\langle P_{\alpha}P_{\beta}P_{\gamma}P_{\delta}P_{\varepsilon}x, x\rangle, \\
\trace\sum\limits_{\beta,\gamma}(h^{\alpha}h^{\beta}h^{\gamma}h^{\beta}h^{\gamma})&=-\sum\limits_{\beta\neq\alpha}\sum\limits_{\gamma\neq\alpha,\beta}\sum\limits_{\delta\neq\alpha,\beta,\gamma}\sum\limits_{\varepsilon\neq\alpha,\beta,\gamma,\delta}\langle P_{\beta}P_{\gamma}P_{\delta}P_{\varepsilon}x, x \rangle\langle P_{\alpha}P_{\beta}P_{\gamma}P_{\delta}P_{\varepsilon}x, x \rangle.
\end{align*}
\end{lem}
\begin{proof}
Using Lemma \ref{h alpha g beta gamma} (2), we obtain
\begin{align*}
(h^{\beta}G)_{ij}&=-\sum\limits_{\gamma}\sum\limits_{\delta\neq\gamma}g_{\beta\gamma\delta}(e_{i})g_{\gamma\delta}(e_{j}),\\
\sum\limits_{\beta}(h^{\beta}Gh^{\beta})_{ij}&=\sum\limits_{\beta}\sum\limits_{\gamma\neq\beta}\sum\limits_{\delta\neq\beta,\gamma}g_{\beta\gamma\delta}(e_{i})g_{\beta\gamma\delta}(e_{j}).
\end{align*}

Using Lemma \ref{h alpha h beta}, this gives
\begin{align*}
\sum_{\beta,\gamma}(h^{\beta}h^{\gamma}h^{\gamma}h^{\beta})_{ij}&=\sum_{\beta}\sum_{k,l}\left((m+1)h^{\beta}_{ik}\delta_{kl}h^{\beta}_{lj}-h^{\beta}_{ik}G_{kl}h^{\beta}_{lj}\right)  \\
&=(m+1)^{2}\delta_{ij}-(m+1)G_{ij}-\sum\limits_{\beta}\sum\limits_{\gamma\neq\beta}\sum\limits_{\delta\neq\beta,\gamma}g_{\beta\gamma\delta}(e_{i})g_{\beta\gamma\delta}(e_{j}).
\end{align*}

Taking the trace with $h^\alpha$, we have
\begin{equation}\label{h^alpha h^beta h^gamma h^gamma h^beta}
\begin{aligned}
\trace\sum_{\beta,\gamma}(h^{\alpha}h^{\beta}h^{\gamma}h^{\gamma}h^{\beta})
&=(m+1)^{2}\sum_{i}h^{\alpha}_{ii}-(m+1)\sum_{i}(h^{\alpha}G)_{ii} \\
&\quad -\sum_{\beta}\sum_{\gamma\neq\beta}\sum_{\delta\neq\beta,\gamma}\sum_{i,k}h^{\alpha}_{ik}g_{\beta\gamma\delta}(e_{k})g_{\beta\gamma\delta}(e_{i}).
\end{aligned}
\end{equation}

Since $M_{+}$ is minimal, $\sum\limits_{i}h^{\alpha}_{ii}=0$. Furthermore, by Lemma \ref{h alpha g beta gamma} (2) and (4), the second term vanishes:
\begin{equation}\label{trace h alpha G}
\begin{aligned}
 \sum_{i}(h^{\alpha}G)_{ii} &=-\sum_{\gamma}\sum_{\delta\neq\gamma}\sum_{i}g_{\alpha\gamma\delta}(e_{i})g_{\gamma\delta}(e_{i}) \\
&=\sum_{\gamma}\sum_{\delta\neq\gamma}\langle P_{\gamma}P_{\delta}P_{\gamma}P_{\delta}P_{\alpha}x, x\rangle \\
&=-\sum_{\gamma}\sum_{\delta\neq\gamma}\langle P_{\alpha}x, x\rangle = 0.
\end{aligned}
\end{equation}

Combining Lemma \ref{h alpha g beta gamma} (3), (4) with   \eqref{P_alpha P_beta P_gamma}, we have
\begin{equation}\label{sum h alpha g 3 g 3}
\begin{aligned}
&\sum\limits_{\beta}\sum\limits_{\gamma\neq\beta}\sum\limits_{\delta\neq\beta,\gamma}\sum\limits_{i,k}h^{\alpha}_{ik}g_{\beta\gamma\delta}(e_{i})g_{\beta\gamma\delta}(e_{k})\\
=&-2\sum\limits_{\beta\neq\alpha}\sum\limits_{\gamma\neq\alpha,\beta}\sum\limits_{\delta\neq\alpha,\beta,\gamma}\sum\limits_{\varepsilon\neq\alpha,\beta,\gamma,\delta}\langle P_{\beta}P_{\gamma}P_{\delta}P_{\varepsilon}x, x\rangle\langle P_{\alpha}P_{\beta}P_{\gamma}P_{\delta}P_{\varepsilon}x, x\rangle. 
\end{aligned} 
\end{equation}

Substituting   \eqref{trace h alpha G} and   \eqref{sum h alpha g 3 g 3} back into   \eqref{h^alpha h^beta h^gamma h^gamma h^beta} yields the first identity:
\begin{align*}
\trace\sum_{\beta,\gamma}(h^{\alpha}h^{\beta}h^{\gamma}h^{\gamma}h^{\beta})&=2\sum\limits_{\beta\neq\alpha}\sum\limits_{\gamma\neq\alpha,\beta}\sum\limits_{\delta\neq\alpha,\beta,\gamma}\sum\limits_{\varepsilon\neq\alpha,\beta,\gamma,\delta}\langle P_{\beta}P_{\gamma}P_{\delta}P_{\varepsilon}x, x\rangle\langle P_{\alpha}P_{\beta}P_{\gamma}P_{\delta}P_{\varepsilon}x, x\rangle.
\end{align*}

Next, applying Lemma \ref{h alpha h beta}, we obtain
\begin{align*}
\sum_{\beta, \gamma} (h^\beta h^\gamma h^\beta h^\gamma)_{kj} &= \sum_{\beta, \gamma}\sum_{l} \langle P_\beta e_k, P_\gamma e_l \rangle \langle P_\beta e_l, P_\gamma e_j \rangle \\
&\quad - \sum_{\beta, \gamma} \sum_{\varepsilon \neq \beta, \gamma}\sum_{l} \langle P_\beta e_k, P_\gamma e_l \rangle g_{\beta \varepsilon}(e_l) g_{\gamma \varepsilon}(e_j) \\
&\quad - \sum_{\beta, \gamma} \sum_{\delta \neq \beta, \gamma}\sum_{l} \langle P_\beta e_l, P_\gamma e_j \rangle g_{\beta \delta}(e_k) g_{\gamma \delta}(e_l) \\
&\quad + \sum_{\beta, \gamma} \sum_{\delta \neq \beta, \gamma} \sum_{\varepsilon \neq \beta, \gamma}\sum_{l} g_{\beta \delta}(e_k) g_{\gamma \delta}(e_l) g_{\beta \varepsilon}(e_l) g_{\gamma \varepsilon}(e_j).
\end{align*}

Using Lemma \ref{h alpha g beta gamma}, we have
\begin{align*}
\sum_{\beta, \gamma} (h^\beta h^\gamma h^\beta h^\gamma)_{kj}
&= (1-m^2)\delta_{kj} + 2(m-1) G_{kj} + \sum_{\beta} \sum_{\gamma \neq \beta} \sum_{\delta \neq \beta, \gamma} g_{\beta\gamma \delta}(e_k) g_{\beta\gamma\delta}(e_j) \\
&\quad - \sum_{\beta} \sum_{\gamma \neq \beta} \sum_{\delta \neq \beta, \gamma} \sum_{\varepsilon \neq \beta, \gamma, \delta} g_{\beta \delta}(e_k) g_{\gamma \varepsilon}(e_j) \langle P_\beta P_\gamma P_\delta P_\varepsilon x, x \rangle.
\end{align*}

Taking the trace with $h^\alpha$, we obtain
\begin{equation*}
\begin{aligned}
\trace\sum_{\beta,\gamma}(h^{\alpha}h^{\beta}h^{\gamma}h^{\beta}h^{\gamma})
&=(1-m^{2})\sum_{i}h^{\alpha}_{ii}+2(m-1)\sum_{i}(h^{\alpha}G)_{ii} \\
&\quad +\sum_{\beta}\sum_{\gamma\neq\beta}\sum_{\delta\neq\beta,\gamma}\sum_{i,k}h^{\alpha}_{ik}g_{\beta\gamma\delta}(e_{k})g_{\beta\gamma\delta}(e_{i}) \\
&\quad -\sum_{\beta} \sum_{\gamma \neq \beta} \sum_{\delta \neq \beta, \gamma} \sum_{\varepsilon \neq \beta, \gamma, \delta}\sum_{i,k}h^{\alpha}_{ik}g_{\beta \delta}(e_k) g_{\gamma \varepsilon}(e_i) \langle P_\beta P_\gamma P_\delta P_\varepsilon x, x \rangle.
\end{aligned}
\end{equation*}

Combining this with Lemma \ref{h alpha g beta gamma},   \eqref{trace h alpha G} and   \eqref{sum h alpha g 3 g 3}, further calculation yields
\begin{align*}
\trace\sum_{\beta,\gamma}(h^{\alpha}h^{\beta}h^{\gamma}h^{\beta}h^{\gamma})=-\sum\limits_{\beta\neq\alpha}\sum\limits_{\gamma\neq\alpha,\beta}\sum\limits_{\delta\neq\alpha,\beta,\gamma}\sum\limits_{\varepsilon\neq\alpha,\beta,\gamma,\delta}\langle P_{\beta}P_{\gamma}P_{\delta}P_{\varepsilon}x, x\rangle\langle P_{\alpha}P_{\beta}P_{\gamma}P_{\delta}P_{\varepsilon}x, x\rangle.
\end{align*}
\end{proof}

We recall the relationship between the Levi-Civita connection $\nabla$ on $M_{+}$ and the standard connection $\nabla^R$ on the Euclidean space. As $M_{+}$ is a submanifold in the sphere and hence in Euclidean space, for any vector fields $e_{i},e_{j}\in\Gamma(TM_{+})$,
\begin{equation}\label{D_e_j e_i}
\nabla^R_{e_{j}}e_{i}=\nabla_{e_{j}}e_{i}-\sum\limits_{\gamma}\langle P_{\gamma}e_{i},e_{j} \rangle P_{\gamma}x-\langle e_{i}, e_{j} \rangle x.
\end{equation}

Furthermore, if we choose a local orthonormal frame $\{e_{1},\dots,e_{n}\}$ such that $\nabla_{e_{i}}e_{j}|_{x}=0$, then at $x$,
\begin{equation}\label{D_e_j(P_alphae_i)}
\nabla^R_{e_{j}}(P_{\alpha}e_{i})=P_{\alpha}(\nabla^R_{e_{j}}e_{i})=-\delta_{ij}P_{\alpha}x-\sum\limits_{\gamma}\langle P_{\gamma}e_{i}, e_{j} \rangle P_{\alpha}P_{\gamma}x.
\end{equation}
It then follows that
\begin{equation}\label{e_j(g_alpha beta)}
\begin{aligned}
e_{j}(g_{\alpha\beta}(e_{i}))&=e_{j}\langle P_{\alpha}P_{\beta}x, e_{i} \rangle \\
&=\langle P_{\alpha}P_{\beta}(\nabla^R_{e_{j}}x), e_{i} \rangle-\delta_{ij}\delta_{\alpha\beta}-\sum_{\gamma}\langle P_{\gamma}e_{i}, e_{j} \rangle\langle P_{\alpha}P_{\beta}x, P_{\gamma}x \rangle \\
&=\langle P_{\alpha}P_{\beta}e_{j}, e_{i} \rangle
-\delta_{ij}\delta_{\alpha\beta}.
\end{aligned}
\end{equation}

\begin{lem}\label{h lie brac ijj}
For the focal submanifold $M_{+}$, let $\alpha \neq \beta$. The second fundamental form satisfies:
\begin{equation*}
(\delta_{\nabla^\perp}\Omega^\perp)_{\alpha\beta}(e_i)=-2(n-3m+2)g_{\alpha\beta}(e_{i})+4\sum_{\gamma\neq\alpha, \beta}\sum_{\delta\neq\alpha, \beta, \gamma}g_{\gamma\delta}(e_{i})\langle P_{\delta}P_{\beta}P_{\alpha}P_{\gamma}x, x \rangle.
\end{equation*}
\end{lem}
\begin{proof}
By the Ricci equation for the normal curvature tensor,
\begin{equation*}
\begin{aligned}
    (\delta_{\nabla^\perp}\Omega^\perp)_{\alpha\beta}(e_i)&=-\sum_j (D_j\Omega^\perp)_{\alpha\beta ji}\\
    &=\sum_j\left(\nabla_j \Omega^\perp_{\alpha\beta}(e_i,e_j)+\sum_{\gamma}\theta_{\gamma\alpha}(e_j)\Omega^\perp_{\gamma\beta}(e_i,e_j)+\sum_{\gamma}\theta_{\gamma\beta}(e_j)\Omega^\perp_{\alpha\gamma}(e_i,e_j)\right)\\
    &=\sum_j\nabla_{j}[h^{\alpha}, h^{\beta}]_{ij}-\sum_{j,\gamma}g_{\alpha\gamma}(e_j)[h^{\gamma}, h^{\beta}]_{ij}+\sum_{j,\gamma}g_{\beta\gamma}(e_j)[h^{\gamma}, h^{\alpha}]_{ij},
\end{aligned}
\end{equation*}
where the last equation holds since
\[\theta_{\alpha\beta}(e_j)=\langle \nabla^R_{e_j} P_\alpha x , P_\beta x \rangle=\langle P_\alpha e_j ,P_\beta x \rangle=g_{\alpha \beta}(e_j).\]
Using Lemma \ref{h alpha h beta}, we have
\begin{equation*}\begin{aligned}
\sum\limits_{j}\nabla_{j}[h^{\alpha}, h^{\beta}]_{ij}&=2\sum\limits_{j}e_{j}\langle P_{\alpha}e_{i}, P_{\beta}e_{j} \rangle-\sum\limits_{j}\sum\limits_{\gamma\neq\alpha,\beta}e_{j}(g_{\alpha\gamma}(e_{i})g_{\beta\gamma}(e_{j})) \\
&\quad +\sum\limits_{j}\sum\limits_{\gamma\neq\alpha,\beta}e_{j}(g_{\beta\gamma}(e_{i})g_{\alpha\gamma}(e_{j})).
\end{aligned}\end{equation*}

By   \eqref{D_e_j(P_alphae_i)} and the condition $\nabla_{e_{i}}e_{j}|_{x}=0$, the first term becomes
\begin{align*}
\sum\limits_{j}e_{j}\langle P_{\alpha}e_{i}, P_{\beta}e_{j} \rangle&=\sum\limits_{j}\langle \nabla^R_{e_{j}}(P_{\alpha}e_{i}), P_{\beta}e_{j} \rangle+\sum\limits_{j}\langle P_{\alpha}e_{i}, \nabla^R_{e_{j}}(P_{\beta}e_{j}) \rangle \\
&=-\sum\limits_{j}\langle \delta_{ij}P_{\alpha}x+\sum\limits_{\gamma}\langle P_{\gamma}e_{i}, e_{j} \rangle P_{\alpha}P_{\gamma}x, P_{\beta}e_{j} \rangle \\
&\quad \ -\sum\limits_{j}\langle \delta_{jj}P_{\beta}x+\sum\limits_{\gamma}\langle P_{\gamma}e_{j}, e_{j} \rangle P_{\beta}P_{\gamma}x, P_{\alpha}e_{i} \rangle \\
&=-(n-1)g_{\alpha\beta}(e_{i})-\sum\limits_{\gamma}\langle (P_{\gamma}e_{i})^{\top},  (P_{\beta}P_{\alpha}P_{\gamma}x)^{\top}\rangle \\
&=-(n-1)g_{\alpha\beta}(e_{i})-\sum\limits_{\gamma}\langle P_{\gamma}e_{i}, P_{\beta}P_{\alpha}P_{\gamma}x \rangle \\
&\quad \ +\sum\limits_{\gamma,\delta}\langle P_{\gamma}e_{i}, P_{\delta}x \rangle\langle P_{\beta}P_{\alpha}P_{\gamma}x, P_{\delta}x \rangle \\
&=-(n-m+2)g_{\alpha\beta}(e_{i})+\sum\limits_{\gamma}\sum\limits_{\delta\neq\gamma}g_{\gamma\delta}(e_{i})\langle P_{\delta}P_{\beta}P_{\alpha}P_{\gamma}x, x \rangle.
\end{align*}

Furthermore, applying   \eqref{e_j(g_alpha beta)} yields
\begin{align*}
\sum\limits_{j}\sum\limits_{\gamma\neq\alpha,\beta}e_{j}(g_{\alpha\gamma}(e_{i})g_{\beta\gamma}(e_{j}))&=\sum\limits_{j}\sum\limits_{\gamma\neq\alpha,\beta}\langle P_{\alpha}P_{\gamma}e_{j}, e_{i} \rangle g_{\beta\gamma}(e_{j})+\sum\limits_{j}\sum\limits_{\gamma\neq\alpha,\beta}g_{\alpha\gamma}(e_{i})\langle P_{\beta}P_{\gamma}e_{j}, e_{j} \rangle \\
&=\sum\limits_{\gamma\neq\alpha,\beta}\langle (P_{\gamma}P_{\alpha}e_{i})^{\top}, (P_{\beta}P_{\gamma}x)^{\top} \rangle  \\
&=\sum\limits_{\gamma\neq\alpha,\beta}\langle P_{\gamma}P_{\alpha}e_{i}, P_{\beta}P_{\gamma}x \rangle-\sum\limits_{\gamma\neq\alpha,\beta}\sum\limits_{\delta}\langle P_{\gamma}P_{\alpha}e_{i}, P_{\delta}x \rangle\langle P_{\beta}P_{\gamma}x, P_{\delta}x \rangle.
\end{align*}

Applying   \eqref{P_alpha P_beta P_gamma}, we note that $\langle P_{\beta}P_{\gamma}x, P_{\delta}x \rangle = 0$ for all $\delta$. Thus, the second sum vanishes and thus
\begin{align*}
\sum\limits_{j}\sum\limits_{\gamma\neq\alpha,\beta}e_{j}(g_{\alpha\gamma}(e_{i})g_{\beta\gamma}(e_{j}))&=\sum\limits_{\gamma\neq\alpha,\beta}\langle e_{i}, P_{\alpha}P_{\gamma}P_{\beta}P_{\gamma}x \rangle \\
&=-(m-1)g_{\alpha\beta}(e_{i}).
\end{align*}

Similarly, interchanging $\alpha$ and $\beta$ in the preceding calculation yields
\begin{equation*}
\sum\limits_{j}\sum\limits_{\gamma\neq\alpha,\beta}e_{j}(g_{\beta\gamma}(e_{i})g_{\alpha\gamma}(e_{j})) = -(m-1)g_{\beta\alpha}(e_{i}) = (m-1)g_{\alpha\beta}(e_{i}).
\end{equation*}

Combining these results, we obtain
\begin{equation*}
\sum\limits_{j}\nabla_{j}[h^{\alpha}, h^{\beta}]_{ij}= -2(n-2m+3)g_{\alpha\beta}(e_{i}) + 2\sum\limits_{\gamma}\sum\limits_{\delta\neq\gamma}g_{\gamma\delta}(e_{i})\langle P_{\delta}P_{\beta}P_{\alpha}P_{\gamma}x, x \rangle.
\end{equation*}

Splitting the sum over $\gamma$ and $\delta$, it follows that
\begin{equation*}
\sum_{j}\nabla_{j}[h^{\alpha}, h^{\beta}]_{ij}= -2(n-2m+1)g_{\alpha\beta}(e_{i})+2\sum_{\gamma\neq\alpha, \beta}\sum_{\delta\neq\alpha, \beta, \gamma}g_{\gamma\delta}(e_{i})\langle P_{\delta}P_{\beta}P_{\alpha}P_{\gamma}x, x \rangle.
\end{equation*}
It follows from Lemma \ref{h alpha g beta gamma} that 
\begin{equation*}
    \begin{aligned}
        \sum_{j,\gamma}g_{\alpha\gamma}(e_j)[h^\gamma,h^\beta]_{ij}&=\sum_{\gamma,j,k} g_{\alpha\gamma}(e_j)h^\gamma_{ik}h^\beta_{kj}-\sum_{\gamma,j,k} g_{\alpha\gamma}(e_j)h^\beta_{ik}h^\gamma_{kj}\\
        &=-\sum_{\gamma,k}g_{\beta\alpha\gamma}(e_k)h^\gamma_{ik}+\sum_{\gamma,k}g_{\gamma\alpha\gamma}(e_k)h^\beta_{ik}\\
        &=-\sum_{\gamma\neq \alpha,\beta}\sum_k g_{\beta\alpha\gamma}(e_k)h^\gamma_{ik}\\
        &=\sum_{\gamma\neq\alpha ,\beta}\langle P_\gamma P_\beta P_\alpha P_\gamma x ,e_i\rangle-\sum_{\gamma\neq\alpha,\beta}\sum_{\delta}g_{\gamma\delta}(e_i)\langle P_\delta P_\beta P_\alpha P_\gamma x,x\rangle\\
        &=-(m-1)g_{\alpha\beta}(e_i)-\sum_{\gamma\neq\alpha,\beta}\sum_{\delta\neq \alpha,\beta,\gamma}g_{\gamma\delta}(e_i)\langle P_\delta P_\beta P_\alpha P_\gamma x,x\rangle.
    \end{aligned}
\end{equation*}
Similarly,
\[ \sum_{j,\gamma}g_{\beta\gamma}(e_j)[h^\gamma,h^\alpha]_{ij}=(m-1)g_{\alpha\beta}(e_i)+\sum_{\gamma\neq\alpha,\beta}\sum_{\delta\neq \alpha,\beta,\gamma}g_{\gamma\delta}(e_i)\langle P_\delta P_\beta P_\alpha P_\gamma x,x\rangle.\]
Combining equations above completes the proof.
\end{proof}

\subsection{Proof of Theorem \ref{exmp-3}}
\begin{proof}
(1) For any distinct indices $\alpha$ and $\beta$, the vector $P_{\alpha}P_{\beta}x$ is tangent to $M_{+}$ at $x$. Applying Lemma \ref{h lie brac ijj} to $P_{\alpha}P_{\beta}x$, we obtain
\begin{equation*}
    (\delta_{\nabla^\perp}\Omega^\perp)_{\alpha\beta}(P_\alpha P_\beta x)=-2(n-3m+2)-4\sum_{\gamma\neq\alpha, \beta}\sum_{\delta\neq\alpha, \beta, \gamma}\langle P_{\delta}P_{\beta}P_{\alpha}P_{\gamma}x, x \rangle^{2}.
\end{equation*}

For the isoparametric focal submanifolds $M_{+}$ of the OT-FKM type with multiplicity pair $(m_{1},m_{2})$, the dimension is $n=m_{1}+2m_{2}$, where $m=m_{1}$. Thus,
\begin{equation*}
n-3m+2=2(m_{2}-m_{1}+1),
\end{equation*}
and it yields $(\delta_{\nabla^\perp}\Omega^\perp)_{\alpha\beta}(P_\alpha P_\beta x)<0$ when $m_2\geqslant m_1$. Consequently, except for the multiplicity pairs 
\[(m_1,m_2)\in\{(2,1),(4,3),(5,2),(6,1),(8,7),(9,6)\},\]
the normal curvature of $M_{+}$ fails to satisfy the classical Yang-Mills equation.

We now examine the exceptional cases above. By Ge, Tang and Yan
\cite[Proposition~4.3]{G-T-Y}, the second fundamental form of $M_{+}$
is parallel when $(m_1,m_2)=(2,1)$ or $(6,1)$, or in the definite
$(4,3)$ case. Thus $Dh=0$, and differentiating the Ricci equation
in \eqref{Gauss-Codazzi-Ricci} gives $D\Omega^\perp=0$.
It follows from \eqref{co-differential operator} that
$\delta_{\nabla^\perp}\Omega^\perp=0$ in these three cases.

In the indefinite $(4,3)$ case, it follows from the proof of Theorem~5.2 in \cite{F-K-M} that there exists a point $x\in M_{+}$ satisfying $P_0P_1P_2P_3x=x$.
Since $n-3m+2=0$, the preceding formula yields
\begin{align*}
(\delta_{\nabla^\perp}\Omega^\perp)_{01}(P_0P_1x)
&=-4\sum_{\gamma\neq 0,1}\sum_{\delta\neq 0,1,\gamma}\langle P_\delta P_1P_0P_\gamma x,x\rangle^2 
\leqslant -4\langle P_0P_1P_2P_3x,x\rangle^2
=-4.
\end{align*}
Hence the normal curvature does not satisfy the Yang-Mills equation $\delta_{\nabla^\perp}\Omega^\perp=0$ in this case.

To deal with the remaining cases, recall the symmetric
$(0,2)$-tensor $G$ defined by
\[
G_{ij}=\sum_{\gamma}\sum_{\delta\neq\gamma}
g_{\gamma\delta}(e_i)g_{\gamma\delta}(e_j)=\sum_{\gamma}\sum_{\delta\neq\gamma}\left\langle
g_{\gamma\delta}(e_i)P_\gamma P_\delta x ,e_j\right\rangle.
\]
Using the induced metric, we also regard $G$ as a symmetric
$(1,1)$-tensor, that is,
\[
G(e_i)=\sum_{\gamma}\sum_{\delta\neq\gamma}
g_{\gamma\delta}(e_i)P_\gamma P_\delta x.
\]
For any $\alpha\neq\beta$, Lemma~\ref{h lie brac ijj} and the
Clifford relations give
\begin{align*}
(\delta_{\nabla^\perp}\Omega^\perp)_{\alpha\beta}(e_i)
&=-4(m_2-m_1+1)g_{\alpha\beta}(e_i)+4\sum_{\gamma\neq\alpha,\beta}
\sum_{\delta\neq\alpha,\beta,\gamma}
g_{\gamma\delta}(e_i)
\langle P_\delta P_\beta P_\alpha P_\gamma x,x\rangle\\
&=-4(m_2-m_1+1)g_{\alpha\beta}(e_i)-4\sum_{\gamma\neq\alpha,\beta}
\sum_{\delta\neq\alpha,\beta,\gamma}
g_{\gamma\delta}(e_i)
\langle P_\gamma P_\delta x,P_\alpha P_\beta x\rangle\\
&=-4(m_2-m_1+1)g_{\alpha\beta}(e_i)-4\sum_{\gamma}\sum_{\delta\neq\gamma}
g_{\gamma\delta}(e_i)
\left\langle P_\gamma P_\delta x,P_\alpha P_\beta x\right\rangle
+8g_{\alpha\beta}(e_i)\\
&=4\left\langle
(m_1-m_2+1)e_i-G(e_i),P_\alpha P_\beta x
\right\rangle.
\end{align*}
Suppose that $\delta_{\nabla^\perp}\Omega^\perp=0$. Then, at every
$x\in M_{+}$ and for each $i$, we have
\begin{align*}
\left(G\circ\bigl((m_1-m_2+1)\mathrm{id}-G\bigr)\right)(e_i)
&=\sum_{\alpha=0}^{m}\sum_{\beta\neq\alpha}
\left\langle
(m_1-m_2+1)e_i-G(e_i),P_\alpha P_\beta x
\right\rangle P_\alpha P_\beta x=0.
\end{align*}
Consequently,
\[
G^2=(m_1-m_2+1)G.
\]
Since $m_1>m_2$ for the remaining multiplicity pairs, set
\[
\widetilde G=\frac{1}{m_1-m_2+1}G.
\]
Then $\widetilde G$ is symmetric and satisfies
$\widetilde G^2=\widetilde G$. Its eigenvalues are therefore $0$ or
$1$, and $\trace(\widetilde G)$ is an integer not exceeding
$n=m_1+2m_2$. On the other hand, since each $P_\gamma P_\delta x$
with $\gamma\neq\delta$ is a unit tangent vector, we obtain
\begin{align*}
\trace(\widetilde G)
&=\frac{1}{m_1-m_2+1}\trace(G)\\
&=\frac{1}{m_1-m_2+1}
\sum_{\gamma}\sum_{\delta\neq\gamma}\sum_i
\langle P_\gamma P_\delta x,e_i\rangle^2\\
&=\frac{1}{m_1-m_2+1}
\sum_{\gamma}\sum_{\delta\neq\gamma}
|P_\gamma P_\delta x|^2\\
&=\frac{m_1(m_1+1)}{m_1-m_2+1}.
\end{align*}

\begin{itemize}
    \item For $(m_1,m_2)=(5,2)$, this gives
    $\trace(\widetilde G)=\frac{15}{2}$, which is not an integer.

    \item For $(m_1,m_2)=(8,7)$, this gives
    $\trace(\widetilde G)=36$, which exceeds $n=22$.

    \item For $(m_1,m_2)=(9,6)$, this gives
    $\trace(\widetilde G)=\frac{45}{2}$, which is again not an integer.
\end{itemize}

These contradictions exclude all the remaining cases.
Thus the normal curvature satisfies the Yang-Mills equation precisely when
$(m_1,m_2)=(2,1)$ or $(6,1)$, or in the definite $(4,3)$ case.

(2) Since $M_{+}$ is a minimal submanifold of the sphere, Theorem \ref{thm-1} implies that $M_{+}$ is Normal-Yang-Mills if and only if
\begin{equation*}
\forall\alpha,\ \sum\limits_{\beta,i,k}D_{i}\left( h^\beta_{ik}(\delta_{\nabla^\perp}\Omega^\perp)_{\alpha\beta}(e_k)\right)+\sum\limits_{\beta,\gamma}{\trace}(h^{\alpha}h^{\beta}h^{\gamma}h^{\beta}h^{\gamma}-h^{\alpha}h^{\beta}h^{\gamma}h^{\gamma}h^{\beta})=0.  
\end{equation*}

Utilizing the minimality of $M_{+}$, we have
\begin{align*}
\sum\limits_{\beta,i,k}D_{i}\left( h^\beta_{ik}(\delta_{\nabla^\perp}\Omega^\perp)_{\alpha\beta}(e_k)\right)&=\sum\limits_{\beta,i,k}(D_{k}h)^{\beta}_{ii}(\delta_{\nabla^\perp}\Omega^\perp)_{\alpha\beta}(e_k)+\sum\limits_{\beta,i,k}h^{\beta}_{ik}D_{i}(\delta_{\nabla^\perp}\Omega^\perp)_{\alpha\beta}(e_k) \\
&=\sum\limits_{\beta,i,k}h^{\beta}_{ik}D_{i}(\delta_{\nabla^\perp}\Omega^\perp)_{\alpha\beta}(e_k)\\
&=\sum\limits_{\beta,i,k}h^{\beta}_{ik}\left(\nabla_{i}(\delta_{\nabla^\perp}\Omega^\perp)_{\alpha\beta}(e_k)+\sum_\gamma\theta_{\gamma\alpha}(e_i)(\delta_{\nabla^\perp}\Omega^\perp)_{\gamma\beta}(e_k)\right.\\
&\hspace{5cm}+\left.\sum_\gamma\theta_{\gamma\beta}(e_i)(\delta_{\nabla^\perp}\Omega^\perp)_{\alpha\gamma}(e_k)\right)
\end{align*}

Combining Lemma \ref{h lie brac ijj} and  \eqref{e_j(g_alpha beta)}, for each $\beta\neq\alpha$, we compute using $\nabla_{e_i}e_j|_x=0$ that
\begin{align*}
&\nabla_{i}(\delta_{\nabla^\perp}\Omega^\perp)_{\alpha\beta}(e_k)\\
=&\nabla_{i}\left((\delta_{\nabla^\perp}\Omega^\perp)_{\alpha\beta}(e_k)\right)\\
=& -2(n-3m+2)e_i\left(g_{\alpha\beta}(e_{k})\right)+4\sum_{\gamma\neq\alpha, \beta}\sum_{\delta\neq\alpha, \beta, \gamma}e_i\big( g_{\gamma\delta}(e_{k})\langle P_{\delta}P_{\beta}P_{\alpha}P_{\gamma}x, x \rangle\big) \\
=& -2(n-3m+2)(\langle P_{\alpha}P_{\beta}e_{i}, e_{k} \rangle-\delta_{ik}\delta_{\alpha\beta})+4\sum_{\gamma\neq\alpha, \beta}\sum_{\delta\neq\alpha, \beta, \gamma}\langle P_{\gamma}P_{\delta}e_{i}, e_{k} \rangle\langle P_{\alpha}P_{\beta}P_{\gamma}P_{\delta}x, x \rangle \\
&-4\sum_{\gamma\neq\alpha, \beta}\sum_{\delta\neq\alpha, \beta, \gamma}\delta_{ik}\delta_{\gamma\delta} \langle P_{\alpha}P_{\beta}P_{\gamma}P_{\delta}x, x \rangle +8\sum_{\gamma\neq\alpha, \beta}\sum_{\delta\neq\alpha, \beta, \gamma}g_{\gamma\delta}(e_{k})\langle P_{\alpha}P_{\beta}P_{\gamma}P_{\delta}x, e_{i} \rangle.
\end{align*}

Since $h_{ik}^{\beta}$ is symmetric in $i$ and $k$, whereas $\langle P_{\alpha}P_{\beta}e_{i}, e_{k} \rangle$ is skew-symmetric, summing their product over $i$ and $k$ yields zero. By minimality and Lemma \ref{h alpha g beta gamma} we have
\begin{equation*}
\begin{aligned}
&\sum\limits_{\beta,i,k}h^{\beta}_{ik}\nabla_{i}(\delta_{\nabla^\perp}\Omega^\perp)_{\alpha\beta}(e_k)\\
=& 8\sum\limits_{\beta\neq\alpha}\sum\limits_{\gamma\neq\alpha, \beta}\sum\limits_{\delta\neq\alpha, \beta, \gamma}\sum\limits_{i,k}h^{\beta}_{ik} g_{\gamma\delta}(e_k)\langle P_{\alpha}P_{\beta}P_{\gamma}P_{\delta}x, e_{i} \rangle \\
=&-8\sum\limits_{\beta\neq\alpha}\sum\limits_{\gamma\neq\alpha, \beta}\sum\limits_{\delta\neq\alpha, \beta, \gamma}\sum\limits_{i}\langle P_{\beta} P_{\gamma}P_{\delta}x, e_{i} \rangle\langle P_{\alpha}P_{\beta}P_{\gamma}P_{\delta}x, e_{i} \rangle \\
=&-8\sum\limits_{\beta\neq\alpha}\sum\limits_{\gamma\neq\alpha, \beta}\sum\limits_{\delta\neq\alpha, \beta, \gamma}\langle (P_{\beta}P_{\gamma}P_{\delta}x)^{\top}, (P_{\alpha}P_{\beta}P_{\gamma}P_{\delta}x)^{\top} \rangle \\
=&-8\sum\limits_{\beta\neq\alpha}\sum\limits_{\gamma\neq\alpha, \beta}\sum\limits_{\delta\neq\alpha, \beta, \gamma} \left( \langle P_{\beta}P_{\gamma}P_{\delta}x, P_{\alpha}P_{\beta}P_{\gamma}P_{\delta}x \rangle - \sum_{\varepsilon} \langle P_{\beta}P_{\gamma}P_{\delta}x, P_{\varepsilon}x \rangle \langle P_{\alpha}P_{\beta}P_{\gamma}P_{\delta}x, P_{\varepsilon}x \rangle \right).
\end{aligned}
\end{equation*}

By the anti-commutativity of the Clifford system, $\langle P_{\beta}P_{\gamma}P_{\delta}x, P_{\alpha}P_{\beta}P_{\gamma}P_{\delta}x \rangle=0$. Furthermore, in the second summation above, $\langle P_{\beta}P_{\gamma}P_{\delta}x, P_{\varepsilon}x \rangle\langle P_{\alpha}P_{\beta}P_{\gamma}P_{\delta}x, P_{\varepsilon}x \rangle$ can only be non-zero when $\varepsilon \notin \{\alpha, \beta, \gamma, \delta\}$. Thus, the summation restricts to $\varepsilon \neq \alpha, \beta, \gamma, \delta$, yielding
\begin{equation*}
\sum\limits_{\beta,i,k}h^{\beta}_{ik}\nabla_{i}(\delta_{\nabla^\perp}\Omega^\perp)_{\alpha\beta}(e_k)= -8\sum\limits_{\beta\neq\alpha}\sum\limits_{\gamma\neq\alpha,\beta}\sum\limits_{\delta\neq\alpha,\beta,\gamma}\sum\limits_{\varepsilon\neq\alpha,\beta,\gamma,\delta}\langle P_{\beta}P_{\gamma}P_{\delta}P_{\varepsilon}x, x \rangle\langle P_{\alpha}P_{\beta}P_{\gamma}P_{\delta}P_{\varepsilon}x, x \rangle.
\end{equation*} 

Besides, it follows from Lemmas \ref{h alpha g beta gamma} and \ref{h lie brac ijj} that
\begin{align*}
    &\sum_{\beta,\gamma,i,k}h^{\beta}_{ik}\theta_{\gamma\alpha}(e_i)(\delta_{\nabla^\perp}\Omega^\perp)_{\gamma\beta}(e_k)\\
    =&\sum_{\beta,\gamma,i,k}h^{\beta}_{ik}g_{\gamma\alpha}(e_i)(\delta_{\nabla^\perp}\Omega^\perp)_{\gamma\beta}(e_k)\\
    =&\sum_{\beta,\gamma,k}g_{\alpha\beta\gamma}(e_k)\left(2(n-3m+2)g_{\gamma\beta}(e_{k})-4\sum_{\delta\neq\gamma, \beta}\sum_{\varepsilon\neq\gamma, \beta, \delta}g_{\delta\varepsilon}(e_{k})\langle P_{\varepsilon}P_{\beta}P_{\gamma}P_{\delta}x, x \rangle\right)\\
    =&-2(n-3m+2)\sum_{\beta,\gamma}\langle P_\alpha P_\beta P_\gamma P_\gamma P_\beta x,x\rangle\\
    &+4\sum\limits_{\beta\neq\alpha}\sum\limits_{\gamma\neq\alpha,\beta}\sum\limits_{\delta\neq\alpha,\beta,\gamma}\sum\limits_{\varepsilon\neq\alpha,\beta,\gamma,\delta}\langle P_{\delta}P_{\varepsilon}P_{\alpha}P_{\beta}P_{\gamma}x, x \rangle\langle P_{\varepsilon}P_{\beta}P_{\gamma}P_{\delta}x, x \rangle\\
    =&-4\sum\limits_{\beta\neq\alpha}\sum\limits_{\gamma\neq\alpha,\beta}\sum\limits_{\delta\neq\alpha,\beta,\gamma}\sum\limits_{\varepsilon\neq\alpha,\beta,\gamma,\delta}\langle P_{\alpha}P_{\beta}P_{\gamma}P_{\delta}P_{\varepsilon}x, x \rangle\langle P_{\beta}P_{\gamma}P_{\delta}P_{\varepsilon}x, x \rangle,
\end{align*}
while
\[\sum_{\beta,\gamma,i,k}h^{\beta}_{ik}\theta_{\gamma\beta}(e_i)(\delta_{\nabla^\perp}\Omega^\perp)_{\alpha\gamma}(e_k)=-\sum_{\beta,\gamma,k}g_{\beta\gamma\beta}(e_k)(\delta_{\nabla^\perp}\Omega^\perp)_{\alpha\gamma}(e_k)=0.\]
Combining the preceding identities with Lemma \ref{trace degree 5}, we obtain that for each $\alpha$,
\begin{multline*}
\sum\limits_{\beta,i,k}D_{i}\left( h^\beta_{ik}(\delta_{\nabla^\perp}\Omega^\perp)_{\alpha\beta}(e_k)\right)+\sum\limits_{\beta,\gamma}{\trace}(h^{\alpha}h^{\beta}h^{\gamma}h^{\beta}h^{\gamma}-h^{\alpha}h^{\beta}h^{\gamma}h^{\gamma}h^{\beta})\\ 
=-15\sum\limits_{\beta\neq\alpha}\sum\limits_{\gamma\neq\alpha,\beta}\sum\limits_{\delta\neq\alpha,\beta,\gamma}\sum\limits_{\varepsilon\neq\alpha,\beta,\gamma,\delta}\langle P_{\beta}P_{\gamma}P_{\delta}P_{\varepsilon}x, x \rangle\langle P_{\alpha}P_{\beta}P_{\gamma}P_{\delta}P_{\varepsilon}x, x \rangle.
\end{multline*}

Therefore,  $M_{+}$ is Normal-Yang-Mills if and only if  for each $\alpha$,
\begin{equation}\label{sum distinct normal indices}
\sum\limits_{\beta\neq\alpha}\sum\limits_{\gamma\neq\alpha,\beta}\sum\limits_{\delta\neq\alpha,\beta,\gamma}\sum\limits_{\varepsilon\neq\alpha,\beta,\gamma,\delta}\langle P_{\beta}P_{\gamma}P_{\delta}P_{\varepsilon}x, x \rangle\langle P_{\alpha}P_{\beta}P_{\gamma}P_{\delta}P_{\varepsilon}x, x \rangle=0, \forall x\in M_+.
\end{equation}

We analyze this condition based on the multiplicity $m$:
\begin{itemize}
    \item \textbf{Case $m\leqslant 3$:} The normal space is at most four dimensions. It is impossible to choose five distinct normal indices. Thus,   \eqref{sum distinct normal indices} vanishes and it follows that $M_+$ is Normal-Yang-Mills.
    
    \item \textbf{Case $m=4$:} The normal space is exactly five dimensions. Without loss of generality, we fix $\alpha=0$. The remaining indices $(\beta, \gamma, \delta, \varepsilon)$ must form a permutation $\sigma \in S_4$ of the set $\{1, 2, 3, 4\}$.
\end{itemize}

Any permutation $\sigma$ produces a sign change $\operatorname{sgn}(\sigma)$ in the $P_{\sigma(1)}P_{\sigma(2)}P_{\sigma(3)}P_{\sigma(4)}$ due to anti-commutativity.   \eqref{sum distinct normal indices}  turns to
\begin{equation}\label{eval distinct normal indices}
\begin{aligned}
\sum\limits_{\sigma \in S_4} &\langle P_{\sigma(1)}P_{\sigma(2)}P_{\sigma(3)}P_{\sigma(4)}x, x \rangle\langle P_{0}P_{\sigma(1)}P_{\sigma(2)}P_{\sigma(3)}P_{\sigma(4)}x, x \rangle \\
&=  24 \langle P_{1}P_{2}P_{3}P_{4}x, x \rangle\langle P_{0}P_{1}P_{2}P_{3}P_{4}x, x \rangle=0, \quad \forall x\in M_+. 
\end{aligned}
\end{equation}

In the definite case, the Clifford system satisfies $P_{0}P_{1}P_{2}P_{3}P_{4}=I_{2l}$. Multiplying both sides by $P_0$ yields $P_{1}P_{2}P_{3}P_{4}=P_{0}$. As a result, $\langle P_{1}P_{2}P_{3}P_{4}x, x \rangle = \langle P_{0}x, x \rangle = 0$, forcing   \eqref{eval distinct normal indices} to vanish. It follows that $M_+$ is Normal-Yang-Mills in the definite classes when $m=4$ and $l=4k$ for any $k\geqslant 2$.

Conversely, in the indefinite case, $Q:=P_0P_1P_2P_3P_4 \neq \pm I_{2l}$ is a symmetric orthogonal matrix commuting with all $P_{\alpha}$ ($\alpha=0,\dots,4$). Moreover, $Q_0:=P_0Q=P_1P_2P_3P_4$ is a symmetric orthogonal matrix commuting with $P_0$ but anti-commuting with $P_{\alpha}$ for $\alpha=1,\dots,4$. Using the representation \eqref{Clifford-sys-alg}, we have $Q=\mathrm{diag}(A,A)$ and $Q_0=\mathrm{diag}(A,-A)$, where $A=E_1E_2E_3\neq \pm I_l$ is a symmetric orthogonal matrix. The focal submanifold $M_+$ in \eqref{def-M+} can then be rewritten as
$$M_+=\{x=(u,v)\in\mathbb{R}^l\oplus\mathbb{R}^l \mid |u|^2=|v|^2=\frac{1}{2}, \langle u, E_{\alpha}v\rangle=0, \alpha=0,1,2,3\}$$
where $E_0:=I_l$. Given any $v\in\mathbb{R}^l$ with $|v|^2=\frac{1}{2}$, let 
$$V:=\mathrm{Span}\{ E_{\alpha}v \mid \alpha=0,1,2,3\},$$
which is a $4$-dimensional linear subspace of $\mathbb{R}^l$.

Now assume $l=8$. Consider the Clifford algebra $\mathcal{Cl}_7$ action on $\mathbb{R}^8$ generated by $\{E_1, \dots, E_7\}$, where $\{ E_1, E_2, E_3 \}$ represents the $\mathcal{Cl}_3$ action on $\mathbb{R}^8$ of the indefinite $(4,3)$ case up to geometric equivalence. Choosing any $v \in \mathbb{R}^8$ with $|v|^2=\frac{1}{2}$, the set $\{E_{\alpha}v \mid \alpha=0, 1, \dots, 7\}$ forms an orthogonal basis of $\mathbb{R}^8$. Any $u \in V^{\bot}$ satisfying $|u|^2=\frac{1}{2}$ can be written as $u = \sum\limits_{\alpha=4}^7 u_{\alpha}E_{\alpha}v$ with $\sum\limits_{\alpha=4}^7 u_{\alpha}^2 = 1$. Since $A = E_1E_2E_3$ anti-commutes with $E_{\alpha}$ for each $\alpha \in \{4, \dots, 7\}$, $E_{\alpha}AE_{\beta}$ is skew-symmetric for $\alpha \neq \beta$ and $E_{\alpha}AE_{\alpha} = A$. Thus,
$$\langle u, Au\rangle = \sum\limits_{\alpha, \beta=4}^7 u_{\alpha}u_{\beta} \langle E_{\alpha}v, AE_{\beta}v \rangle = -\langle v, Av \rangle.$$ It then follows from $Q=\mathrm{diag}(A,A)$ that
\begin{align*}
    \langle P_0P_1P_2P_3P_4 x, x \rangle=\langle Qx,x\rangle=\langle u, Au \rangle+\langle v, Av \rangle=0.
\end{align*}
Therefore, in the indefinite $(4,3)$ case, \eqref{sum distinct normal indices} and \eqref{eval distinct normal indices} hold, showing that $M_+$ is Normal-Yang-Mills.

When $l>8$, if $M_+$ is Normal-Yang-Mills, then \eqref{eval distinct normal indices} implies
\begin{equation}\label{Cond-indef}
    \langle Qx,x\rangle\langle Q_0x,x\rangle=\langle u, Au\rangle^2-\langle v, Av\rangle^2=0
\end{equation}
where $v\in\mathbb{R}^l$ and $u\in V^{\bot}$ satisfy $|v|^2=|u|^2=\frac{1}{2}$.

Since $A\neq\pm I_l$, there exists $v\in\mathbb{R}^l$ with $|v|^2=\frac{1}{2}$ such that $\langle v, Av\rangle=0$. At this specific $v$, \eqref{Cond-indef} holds only if $\langle u, Au\rangle=0$ across all $u\in V^{\bot}$. This means that $V^{\bot}$ is a totally isotropic subspace with respect to the symmetric bilinear form $\langle \cdot, A\cdot\rangle$ on $\mathbb{R}^l$. However, the dimension of a totally isotropic subspace is at most $\frac{l}{2}$ (see \cite[Lemma 6.2]{G-Z}), which contradicts $\mathrm{dim} V^{\bot}=l-4>\frac{l}{2}$ since $l>8$. Thus, \eqref{Cond-indef} cannot hold in general. Consequently, in the indefinite class with $m=4$ and $l>8$, $M_{+}$ is not Normal-Yang-Mills.
\end{proof}

\subsection{Proof of Theorem \ref{exmp-4}}
\begin{proof}
(1) Recall from \eqref{Yang-Mills equation for tangent} that the Riemannian curvature tensor satisfies the classical Yang-Mills equation if and only if its Ricci tensor is Codazzi. According to Tang and Yan \cite{T-Y}, the focal submanifold $M_{+}$ satisfies this condition if and only if $(m_1, m_2) \in \{(2,1), (6,1)\}$, or $M_{+}$ is the homogeneous focal submanifold diffeomorphic to $Sp(2)$ with $(m_1, m_2) = (4,3)$.

(2) By the Gauss equation \eqref{Gauss-Codazzi-Ricci}, we have:
\begin{align*}
\Omega_{ijkl}&=(\theta_k\wedge\theta_l+\sum\limits_{\alpha}\theta_{k\alpha}\wedge\theta_{l\alpha})(e_i, e_j) \\
&=\sum\limits_{p,q}(\delta_{kp}\delta_{lq}+\sum\limits_{\alpha}h^{\alpha}_{kp}h^{\alpha}_{lq})\theta_{p}\wedge\theta_{q}(e_{i}, e_{j})\\  
&=\delta_{ki}\delta_{lj}-\delta_{kj}\delta_{li}+\sum\limits_{\alpha}(h^{\alpha}_{ki}h^{\alpha}_{lj}-h^{\alpha}_{kj}h^{\alpha}_{li}).
\end{align*}

Using Lemma \ref{h alpha h beta}, it follows that:
\begin{align*}
\Ric_{ik}&=\sum\limits_{j}\left(\delta_{ki}\delta_{jj}-\delta_{kj}\delta_{ji}+\sum\limits_{\alpha}(h^{\alpha}_{ki}h^{\alpha}_{jj}-h^{\alpha}_{kj}h^{\alpha}_{ji})\right) \\
&=(n-m-2)\delta_{ik}+\sum\limits_{\alpha}\sum\limits_{\beta\neq\alpha}g_{\alpha\beta}(e_{i})g_{\alpha\beta}(e_{k}).
\end{align*}

By Theorem \ref{thm-2}, $M_+$ is Tangent-Yang-Mills if and only if for any $\alpha$,
\begin{multline}\label{tangent YM variational}
-\sum_{i,j,k}h^{\alpha}_{ij}\nabla_{k}(\nabla_{j}\Ric_{ik}-\nabla_{k}\Ric_{ij})-2\sum_{\beta}\mathrm{Tr}(h^{\alpha}h^{\beta}h^{\beta})\\
+\sum_{\beta,\gamma}\mathrm{Tr}(h^{\alpha}h^{\beta}h^{\gamma})\cdot\mathrm{Tr}(h^{\beta}h^{\gamma})-\sum_{\beta,\gamma}\mathrm{Tr}(h^{\alpha}h^{\beta}h^{\gamma}h^{\beta}h^{\gamma})=0. 
\end{multline}

Fix a point $x \in M_{+}$, and choose a local orthonormal frame $\{e_i\}$ such that $\nabla_{e_{i}}e_{j}|_{x}=0$. Utilizing \eqref{e_j(g_alpha beta)}, we compute the covariant derivatives of the Ricci tensor at $x$:
\begin{align*}
\nabla_{j}\Ric_{ik}-\nabla_{k}\Ric_{ij} &= 2\sum_{\alpha}\sum_{\beta\neq\alpha}\langle P_{\alpha}P_{\beta}e_{j}, e_{k} \rangle g_{\alpha\beta}(e_{i}) \\
&\quad + \sum_{\alpha}\sum_{\beta\neq\alpha}\langle P_{\alpha}P_{\beta}e_{j}, e_{i} \rangle g_{\alpha\beta}(e_{k}) \\
&\quad - \sum_{\alpha}\sum_{\beta\neq\alpha}\langle P_{\alpha}P_{\beta}e_{k}, e_{i} \rangle g_{\alpha\beta}(e_{j}) \\
&=: T_{1} + T_{2} + T_{3}.
\end{align*}

By the Leibniz rule and   \eqref{e_j(g_alpha beta)}, we have
\begin{equation*}
\sum_{k}e_{k}(T_{1}) = 2\sum_{\alpha}\sum_{\beta\neq\alpha} \left( \sum_{k}e_{k}\left(\langle P_{\alpha}P_{\beta}e_{j}, e_{k} \rangle \right) g_{\alpha\beta}(e_{i}) + \sum_{k}\langle P_{\alpha}P_{\beta}e_{j}, e_{k} \rangle \langle P_{\alpha}P_{\beta}e_{k}, e_{i} \rangle \right).
\end{equation*}

Utilizing the minimality of $M_{+}$ and   \eqref{D_e_j e_i}, a direct calculation yields
\begin{equation*}
\sum_{k}e_{k}\langle P_{\alpha}P_{\beta}e_{j}, e_{k} \rangle = (n-1)g_{\alpha\beta}(e_{j}) + \sum_{\gamma}\sum_{k}h_{kj}^{\gamma}g_{\alpha\beta\gamma}(e_{k}).
\end{equation*}

Applying Lemma \ref{h alpha g beta gamma} (3), we obtain
\begin{equation*}
\sum_{\gamma}\sum_{k}h_{kj}^{\gamma}g_{\alpha\beta\gamma}(e_{k}) = -(m-1)g_{\alpha\beta}(e_{j}) + \sum_{\gamma\neq\alpha,\beta}\sum_{\delta\neq\alpha,\beta,\gamma}g_{\gamma\delta}(e_{j})\langle P_{\delta}P_{\alpha}P_{\beta}P_{\gamma}x, x \rangle.
\end{equation*}

Thus,
\begin{equation*}
\sum_{k}e_{k}\langle P_{\alpha}P_{\beta}e_{j}, e_{k} \rangle = (n-m)g_{\alpha\beta}(e_{j}) + \sum_{\gamma\neq\alpha,\beta}\sum_{\delta\neq\alpha,\beta,\gamma}g_{\gamma\delta}(e_{j})\langle P_{\delta}P_{\alpha}P_{\beta}P_{\gamma}x, x \rangle.
\end{equation*}

Using   \eqref{Vector decomposition}, \eqref{Calculation of tangent components} and   \eqref{Calculation of normal components},  we have
\begin{equation*}
\sum_{k}\langle P_{\alpha}P_{\beta}e_{j}, e_{k} \rangle \langle P_{\alpha}P_{\beta}e_{k}, e_{i} \rangle = \langle P_\alpha P_\beta e_j, P_\beta P_\alpha e_i \rangle + g_{\alpha\beta}(e_{i})g_{\alpha\beta}(e_{j}) + \sum_{\gamma}g_{\alpha\beta\gamma}(e_{i})g_{\alpha\beta\gamma}(e_{j}).
\end{equation*}

Summing over $\alpha$ and $\beta\neq\alpha$, it follows that
\begin{equation}\label{div-T1}
\begin{aligned}
\sum_{k}e_{k}(T_{1}) &= -2m(m+1)\delta_{ij} + 2(n-m+1)\sum_{\alpha}\sum_{\beta\neq\alpha}g_{\alpha\beta}(e_{i})g_{\alpha\beta}(e_{j}) \\
&\quad + 2\sum_{\alpha}\sum_{\beta\neq\alpha}\sum_{\gamma\neq\alpha,\beta}g_{\alpha\beta\gamma}(e_{i})g_{\alpha\beta\gamma}(e_{j}) \\
&\quad + 2\sum_{\alpha}\sum_{\beta\neq\alpha}\sum_{\gamma\neq\alpha,\beta}\sum_{\delta\neq\alpha,\beta,\gamma}g_{\alpha\beta}(e_{i})g_{\gamma\delta}(e_{j})\langle P_{\delta}P_{\alpha}P_{\beta}P_{\gamma}x, x \rangle.
\end{aligned}
\end{equation}

By similar calculations, 
\begin{equation}\label{div-T2}
\sum_{k}e_{k}(T_{2}) = 0.
\end{equation}
and
\begin{equation}\label{div-T3}
\begin{aligned}
\sum_{k}e_{k}(T_{3}) &= -m(m+1)\delta_{ij} + (n-m+1)\sum_{\alpha}\sum_{\beta\neq\alpha}g_{\alpha\beta}(e_{i})g_{\alpha\beta}(e_{j}) \\
&\quad + \sum_{\alpha}\sum_{\beta\neq\alpha}\sum_{\gamma\neq\alpha,\beta}g_{\alpha\beta\gamma}(e_{i})g_{\alpha\beta\gamma}(e_{j}) \\
&\quad + \sum_{\alpha}\sum_{\beta\neq\alpha}\sum_{\gamma\neq\alpha,\beta}\sum_{\delta\neq\alpha,\beta,\gamma}g_{\gamma\delta}(e_{i})g_{\alpha\beta}(e_{j})\langle P_{\delta}P_{\alpha}P_{\beta}P_{\gamma}x, x \rangle.
\end{aligned}
\end{equation}

Combining   \eqref{div-T1},   \eqref{div-T2} and   \eqref{div-T3}, and recalling the definition of the tensor $G_{ij}$, we conclude
\begin{align*}
\sum_{k}\nabla_{k}(\nabla_{j}\Ric_{ik}-\nabla_{k}\Ric_{ij})
&=-3m(m+1)\delta_{ij}+3(n-m+1)G_{ij}  \\
&\quad +3\sum_{\alpha}\sum_{\beta\neq\alpha}\sum_{\gamma\neq\alpha,\beta}g_{\alpha\beta\gamma}(e_{i})g_{\alpha\beta\gamma}(e_{j}) \\
&\quad -3\sum_{\alpha}\sum_{\beta\neq\alpha}\sum_{\gamma\neq\alpha,\beta}\sum_{\delta\neq\alpha,\beta,\gamma}g_{\alpha\beta}(e_{i})g_{\gamma\delta}(e_{j})\langle P_{\alpha}P_{\beta}P_{\gamma}P_{\delta}x, x \rangle.
\end{align*}

Using Lemma \ref{h alpha g beta gamma} and   \eqref{trace h alpha G}, we obtain
\begin{align*}
\sum_{i,j,k}h^{\alpha}_{ij}\nabla_{k}(\nabla_{j}\Ric_{ik}-\nabla_{k}\Ric_{ij})
&=-3m(m+1)\sum_{i}h^{\alpha}_{ii}+3(n-m+1)\sum_{i}(h^{\alpha}G)_{ii}  \\
&\quad +3\sum_{\beta}\sum_{\gamma\neq\beta}\sum_{\delta\neq\beta,\gamma}\sum_{i,j}h^{\alpha}_{ij}g_{\beta\gamma\delta}(e_{i})g_{\beta\gamma\delta}(e_{j}) \\
&\quad -3\sum_{\beta}\sum_{\gamma\neq\beta}\sum_{\delta\neq\beta,\gamma}\sum_{\varepsilon\neq\beta,\gamma,\delta}\sum_{i,j}h^{\alpha}_{ij}g_{\beta\gamma}(e_{i})g_{\delta\varepsilon}(e_{j})\langle P_{\beta}P_{\gamma}P_{\delta}P_{\varepsilon}x, x \rangle,  \\
&=-9\sum\limits_{\beta\neq\alpha}\sum\limits_{\gamma\neq\alpha,\beta}\sum\limits_{\delta\neq\alpha,\beta,\gamma}\sum\limits_{\varepsilon\neq\alpha,\beta,\gamma,\delta}\langle P_{\beta}P_{\gamma}P_{\delta}P_{\varepsilon}x, x \rangle\langle P_{\alpha}P_{\beta}P_{\gamma}P_{\delta}P_{\varepsilon}x, x \rangle.
\end{align*}

Similarly,
\begin{equation*}
\sum_{\beta}\mathrm{Tr}(h^{\alpha}h^{\beta}h^{\beta})=\sum_{i,j}h^{\alpha}_{ij}\Bigl((m+1)\delta_{ij}-G_{ij}\Bigr)=(m+1)\sum_{i}h^{\alpha}_{ii}-\sum_{i}(h^{\alpha}G)_{ii}=0,
\end{equation*}

Combining \eqref{Vector decomposition}, \eqref{Calculation of tangent components}, \eqref{Calculation of normal components} and Lemma \ref{h alpha h beta}, we have
\begin{align*}
\trace(h^{\beta}h^{\gamma})&=\sum_{i}\Bigl(\langle P_{\beta}e_{i}, P_{\gamma}e_{i} \rangle-\sum_{\delta\neq\beta,\gamma}g_{\beta\delta}(e_{i})g_{\gamma\delta}(e_{i})\Bigr)\\
&=\sum_{i}\delta_{\beta\gamma}-\sum_{\delta\neq\beta,\gamma}\langle(P_{\beta}P_{\delta}x)^{\top},(P_{\gamma}P_{\delta}x)^{\top} \rangle\\
&=(n-m)\delta_{\beta\gamma}.
\end{align*}

Furthermore,
\begin{equation*}
\sum_{\beta,\gamma}\mathrm{Tr}(h^{\alpha}h^{\beta}h^{\gamma})\cdot\mathrm{Tr}(h^{\beta}h^{\gamma})=(n-m)\sum_{\beta}\trace(h^{\alpha}h^{\beta}h^{\beta})=0.
\end{equation*}

Substituting these into  \eqref{tangent YM variational} and by Lemma \ref{trace degree 5}, the left hand side of \eqref{tangent YM variational} becomes
\begin{equation*}
10\sum\limits_{\beta\neq\alpha}\sum\limits_{\gamma\neq\alpha,\beta}\sum\limits_{\delta\neq\alpha,\beta,\gamma}\sum\limits_{\varepsilon\neq\alpha,\beta,\gamma,\delta}\langle P_{\beta}P_{\gamma}P_{\delta}P_{\varepsilon}x, x \rangle\langle P_{\alpha}P_{\beta}P_{\gamma}P_{\delta}P_{\varepsilon}x, x \rangle,
\end{equation*}
which is the obstruction for Tangent-Yang-Mills as well as for Normal-Yang-Mills in \eqref{sum distinct normal indices}.
Therefore, for OT-FKM type focal submanifolds $M_+$, the Tangent-Yang-Mills and
Normal-Yang-Mills conditions coincide.
\end{proof}

{\bf AI disclosure.}
The mathematical arguments and proofs were developed by the authors. In this revised version, AI tools were used to check calculations and to assist in improving the exposition. We corrected a sign in the definition of $\mathcal{H}^{T}$, revised the expression for $\mathcal{H}^{N}$, and corrected Theorem 1.6(1), with a complete treatment of the remaining cases; all other conclusions remain unchanged. All revisions were reviewed and finalized by the authors. We take full responsibility for its content.


\begin{thebibliography}{99}

\bibitem{B-L}
J. P. Bourguignon, H. B. Lawson, \emph{Stability and isolation phenomena for Yang-Mills fields}, Comm. Math. Physics, \textbf{79(2)}(1981), 189-230.   

\bibitem{CR2015}
T. E. Cecil, P. J. Ryan,
\emph{Geometry of Hypersurfaces}, 
 Springer Monographs in Mathematics, New York: Springer, (2015).


\bibitem{C-Z} 
Q. Chen, Z. R. Zhou, \emph{On gap properties and instabilities of p-Yang-Mills fields}, Canad. J. Math., \textbf{59(6)}(2007), 1245-1259.

\bibitem{C-S-S}   
S. S. Chern, \emph{Topics in differential geometry: selected papers $\&$ lectures of Shiing-Shen Chern}, (2016).

\bibitem{Chi18}
Q. S. Chi,  \emph{The Isoparametric Story, a Heritage of \'{E}lie Cartan}, Proceedings of the International Consortium of Chinese Mathematicians, (2018), 197-260, International Press of Boston, (2020).

\bibitem{F-K-M}
D. Ferus, H. Karcher, H. F. Münzner, \emph{Cliffordalgebren und neue isoparametrische Hyperflächen}, Math. Z., \textbf{177(4)} (1981), 479-502. 

\bibitem{F-U}
D. S. Freed, K. K. Uhlenbeck, \emph{Instantons and four-manifolds}, Springer Science $\&$ Business Media, (2012).

\bibitem{GQTY2025}
J. Q. Ge, C. Qian, Z. Z. Tang, W. J. Yan, 
\emph{An overview of the development of isoparametric theory (in Chinese)}, 
 Sci. Sin. Math., \textbf{55}(2025), 145–168.

\bibitem{G-T-Y}
J. Q. Ge, Z. Z. Tang, W. J. Yan,
\emph{Normal scalar curvature inequality on the focal submanifolds
of isoparametric hypersurfaces},
Int. Math. Res. Not. IMRN, \textbf{2020(2)} (2020), 422-465.


\bibitem{G-Z}
J. Q. Ge, Y. Zhou, \emph{Austere matrices, austere submanifolds and Dupin hypersurfaces}, Adv. Math., \textbf{482} (2025), 110645.

\bibitem{J-Z}
G. Y. Jia, Z. R. Zhou, \emph{Gaps of F-Yang-Mills fields on submanifolds}, Tsukuba J. Math., \textbf{36(1)}(2012), 121-134.

\bibitem{R-S-S}
G. Rudolph, M. Schmidt, \emph{Differential geometry and mathematical physics}, Berlin: Springer, (2012).

\bibitem{S}
K. Sakamoto, \emph{Variational problems of normal curvature tensor and concircular scalar fields}, Tohoku Math. J., Second Series, \textbf{55(2)}(2003), 207-254.

\bibitem{T-Y}
Z. Z. Tang, W. J. Yan, \emph{Isoparametric foliation and a problem of Besse on generalizations of Einstein condition}, Adv. Math., \textbf{285}(2015), 1970-2000.

\bibitem{T-Y 1}
Z. Z. Tang, W. J. Yan, \emph{New examples of Willmore submanifolds in the unit sphere via isoparametric functions}, Ann. Glob. Anal. Geom., \textbf{42(3)}(2012), 403-410.

\bibitem{T}
G. Tian, \emph{Gauge theory and calibrated geometry, I}, Annals of Math., \textbf{151(1)}(2000), 193-268.

\bibitem{W}
C. P. Wang, \emph{Moebius geometry of submanifolds in $S^n$}, Manuscripta Math., \textbf{96(4)}(1998), 517-534.

\bibitem{X}
Y. Q. Xie, \emph{Willmore submanifolds in the unit sphere via isoparametric functions}, Acta Math. Sinica, English Series, \textbf{31(12)}(2015), 1963-1969.
\end{thebibliography}
\end{document}